\documentclass[reqno]{amsart}

\usepackage{amssymb,amsmath,amscd,amsthm}

\usepackage{mathrsfs}
\usepackage{latexsym}
\usepackage{graphicx}
\usepackage{breakcites}

\newtheoremstyle{myplain}{}{}{\it}
{0pt}{\scshape}{}{ }{\thmname{#1}\thmnumber{ #2}\thmnote{ (#3)}}
\newtheoremstyle{mydefinition}{}{}{}
{0pt}{\scshape}{}{ }{\thmname{#1}\thmnumber{ #2}\thmnote{ (#3)}}

\theoremstyle{myplain}

    \newtheorem{Def}{Definition}[section]
        \newtheorem{Lem}[Def]{Lemma}

        \newtheorem{theo}[Def]{Theorem}
        \newtheorem{prop}[Def]{Proposition}
        \newtheorem{rem}[Def]{Remark}
        \newtheorem{hyp}[Def]{Hypothesis}
        \newtheorem{cor}[Def]{Corollary}

\renewcommand{\qed}{\nopagebreak\hfill$\Box$}

\newcommand{\skp}[2]{\mbox{$\left\langle #1\, , \, #2\right\rangle$}}
\newcommand{\skpd}[2]{\mbox{$\left\langle #1\, ,\,#2\right\rangle_{\ell^2}$}}

\newcommand{\natop}[2]{\genfrac{}{}{0pt}{}{#1}{#2}}
\newcommand{\expord}[1]{O\left(e^{-\frac{#1}{\ep}}\right)}

\DeclareMathOperator{\dist}{dist}
\DeclareMathOperator{\Span}{span}

\DeclareMathOperator{\supp}{supp}

\DeclareMathOperator{\spec}{spec}

\DeclareMathOperator{\Op}{Op}

\DeclareMathOperator{\diag}{diag}

\newcommand{\id}{\mathbf{1}}

\numberwithin{equation}{section}

\newcommand{\beqa}{\begin{eqnarray*}}
\newcommand{\eeqa}{\end{eqnarray*}}
\renewcommand{\hat}{\widehat}
\newcommand{\bauf}{\begin{itemize}}
\newcommand{\eauf}{\end{itemize}}
\newcommand{\ben}{\begin{enumerate}}
\newcommand{\een}{\end{enumerate}}
\newcommand{\ra}{\rightarrow}

\renewcommand{\O}{\Omega}
\newcommand{\ep}{\varepsilon}

\newcommand{\R}{{\mathbb R} }
\newcommand{\Z}{{\mathbb Z}}
\newcommand{\C}{{\mathbb C}}
\newcommand{\N}{{\mathbb N}}
\newcommand{\T}{{\mathbb T}}
\newcommand{\Hd}{{\mathbb H}}

\newcommand{\disk}{(\varepsilon {\mathbb Z})^d}

\newcommand{\Ce}{\mathscr C}

\newcommand{\De}{\mathscr D}
\newcommand{\E}{{\mathcal E}}
\newcommand{\F}{{\mathcal F}}

\title{Tunneling for a class of difference operators:\\ Complete Asymptotics}

\author{Markus Klein \and Elke Rosenberger}

\address{
Universit\"at Potsdam\\ Institut f\"ur
Mathematik \\ Am Neuen Palais 10\\ 14469 Potsdam}
\email{mklein@math.uni-potsdam.de, erosen@uni-potsdam.de}

\date{\today}


\begin{document}

\begin{abstract}
We analyze a general class of difference operators $H_\ep = T_\ep
+ V_\ep$ on $\ell^2(\disk)$, where $V_\ep$ is a multi-well
potential and $\ep$ is a small parameter. We
derive full asymptotic expansions of the prefactor of the exponentially small eigenvalue splitting due to interactions between two ``wells'' (minima)
of the potential energy, i.e., for the discrete tunneling effect. We treat both the case where there is a single minimal geodesic (with respect to the
natural Finsler metric induced by the leading symbol $h_0(x,\xi)$ of $H_\ep$) connecting the two minima and the case where the minimal geodesics form
an $\ell+1$ dimensional manifold, $\ell\geq 1$. These results on the tunneling problem are as sharp as the classical results for 
the Schr\"odinger operator in \cite{hesjo}. Technically, our approach is pseudodifferential and we adapt techniques from \cite{hesjo2} and \cite{hepar}
to our discrete setting.
\end{abstract}

\maketitle

\section{Introduction}

The aim of this paper is to derive complete asymptotic expansions for the interaction between two potential minima of a
difference operator on a scaled lattice, i.e., 
for the discrete tunneling effect.

We consider a rather general class of families of
difference operators $\left(H_\ep\right)_{\ep>0}$ on the Hilbert space $\ell^2(\disk)$, as the small parameter  $\ep>0$
tends to zero. The operator $H_\ep$ is given by
\begin{align} \label{Hepein}
H_\ep &= T_\ep + V_\ep ,  \quad\text{where}\quad
T_\ep  = \sum_{\gamma\in\disk} a_\gamma \tau_\gamma ,\\
(\tau_\gamma u)(x) &= u(x+\gamma) \, ,\qquad \quad (a_\gamma u)(x) := a_\gamma(x; \ep) u(x) \quad \mbox{for} 
\quad x,\gamma\in\disk \label{agammataugamma}
\end{align}
and $V_\ep$ is a multiplication operator which in leading order is given by a multiwell-potential
$V_0 \in \Ce^\infty (\R^d)$.

The interaction between neighboring
potential wells leads by means of the tunneling effect
to the fact that
the eigenvalues and eigenfunctions are different from those
of an operator with decoupled wells, which is realized by the direct sum of ``Dirichlet-operators''
situated at the several wells. Since the interaction is small, it can be treated as a perturbation of
the decoupled system.

In \cite{kleinro4}, we showed that it is possible to approximate the
eigenfunctions of the original Hamiltonian $H_\ep$ with respect to a fixed spectral interval
by (linear combinations of) the eigenfunctions of the
several Dirichlet operators situated at the different wells and we gave
a representation of $H_\ep$ with respect to a basis of Dirichlet-eigenfunctions.

In \cite{kleinro5} we gave estimates for the weighted $\ell^2$-norm of the difference between 
exact Dirichlet eigenfunctions and approximate 
Dirichlet eigenfunctions, which are constructed using the WKB-expansions given in \cite{kleinro3}.

In this paper, we consider the special case, that only Dirichlet operators at two wells have 
an eigenvalue (and exactly one) inside a given
spectral interval. Then it is possible to compute complete asymptotic expansions for the elements of the interaction matrix and
to obtain explicit formulae for the leading order term.\\

This paper is based on the thesis \cite{thesis}. It is the sixth in a series of papers 
(see \cite{kleinro} - \cite{kleinro5}); the aim is to develop an analytic approach
to the semiclassical eigenvalue problem and tunneling for $H_\ep$ which is comparable  in detail and
precision
to the well known analysis for the Schr\"odinger operator (see \cite{Si1} and
\cite{hesjo}). We remark that the analysis of tunneling has been extended to classes of pseudodifferential
operators in $\R^d$ in \cite{hepar} where tunneling is discussed for the Klein-Gordon and Dirac operator. 
This article in turn relies heavily on the ideas in the analysis of Harper's equation in \cite{hesjo2} and
previous results from \cite{sjo} covering classes of analytic symbols. Since our formulation of the spectral 
problem  for the operator in \eqref{Hepein} is pseudo-differential in spirit, it has been possible to adapt the
methods of \cite{hepar} to our case. Since our symbols are analytic only in the momentum variable $\xi$, but not in the 
space variable $x$, the results of \cite{sjo} do not all automatically apply.

Our motivation comes from
stochastic problems (see \cite{kleinro}, \cite{begk1}, \cite{begk2}). 
A large class of
discrete Markov chains analyzed in \cite{begk2}
with probabilistic
techniques falls into the framework of difference operators treated in this article.

We expect that similar results hold in the more general case that the Hamiltonian is a generator of a
jump process in $\R^d$, see \cite{klr} for first results in this direction.

\begin{hyp}\label{hyp1}
\begin{enumerate}
\item The coefficients $a_\gamma(x; \ep)$ in \eqref{Hepein} are
functions
\begin{equation}\label{agammafunk}
a: \disk \times \R^d \times (0,\ep_0] \ra \R\, , \qquad (\gamma, x,
\ep) \mapsto a_\gamma(x; \ep)\, ,
\end{equation}
satisfying the following conditions:
\ben
\item[(i)] They have an
expansion
\begin{equation}\label{agammaexp}
a_\gamma(x; \ep) = \sum_{k=0}^{N-1} \ep^k a_\gamma^{(k)}(x)  + R^{(N)}_\gamma (x; \ep)\, ,\qquad N\in\N^*\, ,
\end{equation}
where $a_\gamma \in\Ce^\infty(\R^d\times (0,\ep_0])$ and $a_\gamma^{(k)}\in\Ce^\infty(\R^d)$ for
all $\gamma\in\disk$ and $0\leq k \leq N-1$.
\item[(ii)] $\sum_{\gamma\in\disk} a_\gamma^{(0)}  = 0$ and $a_\gamma^{(0)}
\leq 0$ for $\gamma \neq 0$.
\item[(iii)] $a_\gamma(x; \ep) =
a_{-\gamma}(x+\gamma; \ep)$ for all $x \in \R^d, \gamma \in \disk$
\item[(iv)] For any $c>0$ and $\alpha\in\N^d$ there exists $C>0$ such that for $0\leq k\leq N-1$ uniformly
with respect to $x\in\R^d$ and $\ep\in (0,\ep_0]$.
\begin{equation}\label{abfallagamma}
\| \, e^{\frac{c|.|}{\ep}} \partial_x^\alpha a^{(k)}_.(x)\|_{\ell_\gamma^2(\disk)}\leq C \qquad\text{and} \qquad
\|\,
 e^{\frac{c|.|}{\ep}} \partial_x^\alpha R^{(N)}_.(x)\|_{\ell^2_\gamma(\disk)}
 \leq C\ep^N\; .
\end{equation}
\item[(v)]
$\Span \{\gamma\in\disk\,|\, a^{(0)}_\gamma(x) <0\}= \R^d$ for all $x\in\R^d$.
\een
\item
\ben
\item[(i)] The potential energy $V_\ep$ is the restriction to $\disk$ of a
function $\hat{V}_\ep\in\Ce^\infty (\R^d, \R)$ which has an expansion
\begin{equation}\label{hatVell}
\hat{V}_{\ep}(x) = \sum_{\ell=0}^{N-1}\ep^l V_\ell(x) + R_{N}(x;\ep)  \, ,\qquad N\in\N^*\, ,
\end{equation}
where $V_\ell\in\Ce^\infty(\R^d)$, $R_{N}\in \Ce^\infty (\R^d\times (0,\ep_0])$ for some $\ep_0>0$ and
for any compact set $K\subset \R^d$
there exists a constant $C_K$ such that $\sup_{x\in K} |R_{N}(x;\ep)|\leq C_K \ep^{N}$.
\item[(ii)]
$V_\ep$ is polynomially bounded and there exist constants $R, C > 0$ such that
$V_\ep(x) > C$ for all $|x| \geq R$ and $\ep\in(0,\ep_0]$.
\item[(iii)]
$V_0(x)\geq 0$ and it takes the value $0$ only at a finite number of non-degenerate minima
$x^j,\; j\in \mathcal{C} =\{1,\ldots , r\}$,
which we call potential wells.
\een
\een
\end{hyp}

We remark that for $T_\ep$ defined in
\eqref{Hepein}, under the assumptions given in Hypothesis \ref{hyp1}, one has $T_\ep = \Op_\ep^{\T}(t(.,.;\ep))$ 
(see Appendix \ref{app1} for definition and details of the quantization on the 
$d$-dimensional torus $\T^d := \R^d/ (2\pi \Z)^d$)  where
$t\in\Ce^\infty\left(\R^d\times\T^d\times (0,\ep_0]\right)$ is given by
\begin{equation}\label{talsexp}
t(x,\xi; \ep) = \sum_{\gamma\in\disk} a_\gamma (x; \ep) \exp
\left(-\tfrac{i}{\ep}\gamma\cdot\xi\right)\; .
\end{equation}
Here $t(x,\xi;\ep)$ is considered as a function on $\R^{2d}\times (0,\ep_0]$, which is
$2\pi$-periodic with respect to $\xi$. By condition (a)(iv) in Hypothesis \ref{hyp1}, 
the function $\xi\mapsto t(x, \xi; \ep)$ has
an analytic continuation to $\C^d$. 
Moreover for all $B>0$
\begin{equation}\label{agammasum} 
\sum_\gamma \left| a_\gamma(x; \ep)\right| e^{\frac{B|\gamma|}{\ep}} \leq C \qquad\text{and thus}\qquad
\sup_{x\in\R^d} |a_\gamma(x; \ep)| \leq C e^{-\frac{B|\gamma|}{\ep}}
\end{equation}
uniformly with respect to $x$ and $\ep$.
We further remark that (a)(iv) implies $\bigl|a_\gamma^{(k)}(x)- a_\gamma^{(k)}(x + h)\bigr|\leq C |h|$ for $0\leq k \leq N-1$
uniformly with respect to $\gamma\in\disk$ and $x,h\in\R^d$ and (a)(ii),(iii),(iv) imply that $T_\ep$ is symmetric and 
bounded and that for some $C>0$
\begin{equation}\label{Tvonunten}
 \skpd{u}{T_\ep u} \geq - C \ep \|u\|^2_{\ell^2}\;, \qquad u\in\ell^2(\disk)\; .
\end{equation}

Furthermore, we set
\begin{align}\label{texpand}
t(x,\xi;\ep) &= \sum_{k=0}^{N-1} \ep^k t_k (x,\xi) + \hat{t}_N(x,\xi;\ep)\quad
\text{with}\\
t_k(x, \xi) &:= \sum_{\gamma\in\disk} a_\gamma^{(k)}(x) e^{-\frac{i}{\ep}\gamma\xi}\, ,
\qquad 0\leq k \leq N-1\,,\nonumber\\
\hat{t}_N(x, \xi; \ep) &:= \sum_{\gamma\in\disk} R_\gamma^{(N)}(x; \ep)
e^{-\frac{i}{\ep}\gamma\xi}\nonumber\; .
\end{align}
Thus, in leading order, the symbol of $H_\ep$ is $h_0:=t_0+V_0$.
Combining \eqref{agammaexp} and (a)(iii) shows that the $2\pi$-periodic function $\R^d\ni\xi\mapsto t_0(x,\xi)$ 
is even with respect to $\xi\mapsto -\xi$, i.e.,
\begin{equation}\label{agammasym}
a_\gamma^{(0)}(x) = a_{-\gamma}^{(0)}(x)\, , \qquad x\in\R^d, \,\gamma\in\disk
\end{equation}
 (see \cite{kleinro}, Lemma 1.2) and 
therefore
\begin{equation}\label{tcosh} 
t_0(x,\xi) = \sum_{\gamma\in\disk} a_\gamma^{(0)}(x) \cos \bigl(\tfrac{1}{\ep}\gamma\cdot\xi\bigr) \; . 
\end{equation}
At $\xi=0$, for fixed $x\in\R^d$ the function $t_0$ defined in \eqref{texpand} has by Hypothesis \ref{hyp1}(a)(ii) an expansion
\begin{equation}\label{kinen}
t_0(x,\xi) = \skp{\xi}{B(x)\xi} + \sum_{\natop{|\alpha|=2n}{n\geq2}} B_\alpha (x) \xi^\alpha \qquad\text{as}\;\; |\xi|\to 0
\end{equation}
where $\alpha\in\N^d$, $B\in\Ce^\infty (\R^d, \mathcal{M}(d\times d,\R))$, for any $x\in\R^d$ the matrix $B(x)$  is
positive definite and symmetric and $B_\alpha$ are real functions. By straightforward calculations one gets
for $1\leq \mu,\nu\leq d$
\begin{equation}\label{Bnumu}
B_{\nu\mu}(x) = -\frac{1}{2\ep^2} \sum_{\gamma\in\disk} a_\gamma^{(0)}(x) \gamma_\nu\gamma_\mu\; .
\end{equation}

We set 
\begin{equation}\label{h0tilde}
\tilde{h}_0: \R^{2d} \rightarrow \R\, , \quad \tilde{h}_0(x,\xi) := - t_0(x, i\xi) - V_0(x)\; .
\end{equation}

In order to work in the context of \cite{kleinro2}, we shall assume

\begin{hyp}\label{hyp2}
At the minima $x^j,\, j\in\mathcal{C}$, of $V_0$, we assume that $t_0$ defined in \eqref{texpand} fulfills
\[ t_0(x^j, \xi) >0  \quad\text{if} \quad \xi \in\T^d\setminus \{0\} \, .\]
\end{hyp}

For any set $D\subset \R^d$, we denote the restriction to the lattice by  $D_\ep := D\cap \disk$.\\

By Hypothesis \ref{hyp1}, $\tilde{h}_0$ is even and hyperconvex\footnote{For a normed vector space $V$ we call a function $L\in\Ce^2(V,\R)$
hyperconvex, if there exists a constant $\alpha > 0$ such that
\[ D^2L|_{v_0}(v,v)\geq\alpha \|v\|^2\quad\text{for all}\quad v_0,v\in V\, .\] }
with respect to momentum. 
We showed in \cite{kleinro}, Prop. 2.9, that any function $f\in\Ce^\infty(T^*M,\R)$, which is hyperconvex in each fibre, is automatically
hyperregular\footnote{We recall from e.g. \cite{abma} that $f$ is hyperregular if its fibre derivative $D_Ff$ - related to
the Legendre transform - is a global diffeomorphism: $T^*M \rightarrow TM$.} (here $M$ denotes a smooth manifold, which in our context
is equal to $\R^d$).

We can thus
introduce  the associated Finsler distance $d = d_\ell$ on $\R^d$ as in \cite{kleinro}, Definition 2.16, where we set 
$\widetilde{M}:=\R^d\setminus\{x^k\, , \, k\in\mathcal{C}\}$.
Analog to \cite{kleinro}, Theorem 1.6, it can be shown that $d$ is locally Lipschitz and that for any $j\in\mathcal{C}$, the distance 
$d^j(x):= d(x, x^j)$ 
fulfills the generalized eikonal equation and inequality respectively 
\begin{align}\label{eikonal2} 
\tilde{h}_0 \bigl(x, \nabla d^j (x)\bigr) &= 0 \; ,\qquad x\in\Omega^j \\
\tilde{h}_0\bigl(x, \nabla d^j(x)\bigr) &\leq 0\; , \qquad x\in\R^d
\end{align}
where $\Omega^j$ is some neighborhood of $x_j$.
We remark that, assuming only Hypothesis \ref{hyp1},
it is possible that balls of finite radius with respect to the Finsler distance, i.e. $B_r(x):=\{y\in\R^d\,|\, d(x,y)\leq r\}, r<\infty$, 
are unbounded in the Euclidean distance (and thus
not compact). In this paper, we shall not discuss consequences of this effect.

Crucial quantities for the subsequent analysis are for $j,k\in \mathcal{C}$
\begin{equation}\label{distances}
S_{jk}:= d(x^j, x^k)\, , \qquad\text{and}\qquad  S_0:= \min_{j\neq k}d(x^j,x^k) \; .
\end{equation}

\begin{rem}\label{remhypmultwell}
Since $d$ is locally  Lipschitz-continuous (see \cite{kleinro}), it follows from \eqref{agammasum} that for any
$B>0$ and any bounded region $\Sigma\subset \R^d$ there exists a constant
$C>0$ such that
\begin{equation}\label{agammasupnorm2}
\sum_{\natop{\gamma\in\disk}{|\gamma|<B}} \Bigl\|a_\gamma (\,.\, ; \ep)
e^{\frac{d(.,.+\gamma)}{\ep}}\Bigr\|_{\ell^\infty(\Sigma)} \leq C\; .
\end{equation}
\end{rem}

For $\Sigma\subset\R^d$ we define the space
$\ell^2_{\Sigma_\ep}:= i_{\Sigma_\ep} \left(\ell^2(\Sigma_\ep)\right) \subset \ell^2(\disk)$ where
$ i_{\Sigma_\ep}$ denotes the embedding via zero extension.
Then we define the Dirichlet operator
\begin{equation} \label{HepD}
H_\ep^{\Sigma} :=\id_{\Sigma_\ep} H_\ep|_{\ell^2_{\Sigma_\ep}}  \;:\; \ell^2_{\Sigma_\ep} \rightarrow \ell^2_{\Sigma_\ep}
\end{equation}
with domain $\De (H_\ep^{\Sigma}) = \{u\in\ell^2_{\Sigma_\ep}\,|\, V_\ep u \in \ell^2_{\Sigma_\ep}\}$.

For a fixed spectral interval it is shown in \cite{kleinro4} that the difference
between the exact spectrum and the spectra of Dirichlet
realizations of $H_\ep$ near the different wells is
exponentially small and determined by the Finsler distance between the two nearest neighboring wells.
In the following we give additional assumptions.

The following hypothesis gives assumptions concerning the separation of the different wells using Dirichlet operators and the
restriction to some adapted spectral interval $I_\ep$.

\begin{hyp}\label{hypIMj}
\ben
\item There exist constants $\eta>0$ and $C>0$ such that for all $x\in\R^d$
\[ \left\| a_{(.)} (x; \ep) e^{\frac{1}{\ep} d(x, x+\, . \, )} |\, . \, 
|^{\frac{d + \eta}{2}} \right\|_{\ell^2(\disk)} \leq C \; . \]
\item For $j\in\mathcal{C}$, we choose a
compact manifold $M_j\subset \R^d$ with $\Ce^2$-boundary such that the following holds:
\ben
\item $x^j\in M_j$, $d^j\in\mathscr{C}^2(M_j)$ and $x^k\notin M_j$ for $k\in\mathcal{C}, \,k\neq j$.
\item Let $X_{\tilde{h}_0}$ denote the Hamiltonian vector field with respect to $\tilde{h}_{0}$ defined in \eqref{h0tilde}, 
$F_{t}$ denote the flow of $X_{\tilde{h}_0}$ and set
\begin{equation}\label{Lambdaplus}
\Lambda_{\pm} := \bigl\{ (x,\xi)\in T^*\R^{d}\, |\, F_{t}(x,\xi) \rightarrow (x^j,0)\quad \text{for}
\quad t \rightarrow \mp \infty \bigr\}  \; .
\end{equation}
Then, for $\pi : T^*\R^d \rightarrow \R^d$ denoting the bundle projection $\pi (x, \xi) = x$, we have 
\[ \Lambda_+(M_j):=\pi^{-1}(M_j) \cap \Lambda_+ = \bigl\{ (x, \nabla d^j(x)) \in T^*\R^d\, |\, x\in M_j\bigr\} \; . \]
Moreover $\pi\bigl(F_t(x,\xi)\bigr) \in M_j$ for all $(x,\xi)\in \pi^{-1}(M_j) \cap \Lambda_+$ and all $t\leq 0$. 
\een
\item
Given $M_j,\,  j\in\mathcal{C}$, let $I_\ep = [\alpha (\ep),\beta (\ep)]$ be an interval, such that
$\alpha (\ep),\beta (\ep) = O(\ep)$ for $\ep\to 0$. Furthermore there
exists a function $a(\ep)>0$ with the property $|\log a(\ep)| =
o\left(\frac{1}{\ep}\right),\, \ep\to 0$, such that none of the
operators $H_\ep,H_\ep^{M_1},\ldots H_\ep^{M_r}$ has spectrum in
$[\alpha(\ep)-2a(\ep),\alpha(\ep))$ or
$(\beta(\ep),\beta(\ep)+2a(\ep)]$.
\een
\end{hyp}

By \cite{kleinro}, Theorem 1.5, the base integral curves of $X_{\tilde{h}_0}$ on
$\R^d\setminus\{x_1,\ldots x_m\}$ with energy $0$ are geodesics with respect to
$d$ and vice versa.
Thus Hypothesis \ref{hypIMj}, 2(b), implies in particular that there is a unique minimal geodesic
between any point in $M_j$ and $x^j$.

Clearly, $\Lambda_+ (M_j)$ is a Lagrange manifold (by 2(b)) and since the flow $F_t$ preserves $\tilde{h}_0$, we have
$\Lambda_+(M_j)\subset \tilde{h}_0^{-1}(0)$ by \eqref{Lambdaplus}. Thus the eikonal equation 
$\tilde{h}_0(x, \nabla d^j(x)) =0$ holds for $x\in M_j$.
It follows from the construction of the solution of the
eikonal equation in \cite{kleinro3} that in fact $d^j\in \Ce^\infty(M_j)$. We recall that, in a small neighborhood of $x^j$, the
equation $\xi = \nabla d^j$ parametrizes by construction the outgoing manifold $\Lambda_+$ of the hyperbolic
fixed point $(x^j, 0)$ of $X_{\tilde{h}_0}$ in $T^*M_j$. Hypothesis \ref{hypIMj}, (2), ensures this globally.\\

Since the main theorems in this paper treat fine asymptotics for the interaction between two wells, we assume the following hypothesis.
It guarantees that neither the wells are to far from each other nor the difference between the 
Dirichlet eigenvalues is to big (otherwise the main term of the interaction matrix has the same
order of magnitude as the error term).\\

Given Hypothesis \ref{hypIMj}, we assume in addition

\begin{hyp}\label{hypkj}
\ben
\item
Only two Dirichlet operators $H_\ep^{M_j}$ and 
$H_\ep^{M_k}$, $j,k\in\mathcal{C},$ have an 
eigenvalue (and exactly one) in
the spectral interval $I_\ep$, which we denote by 
$\mu_j$ and $\mu_k$ respectively, with corresponding real Dirichlet eigenfunctions
$v_j$ and $v_k$.
\item We choose coordinates such that $x^j_d<0$ and $x^k_d>0$ and we set 
\begin{equation}\label{Hnull} 
\Hd_d := \{ x\in\R^d\, |\, x_d=0\} \; .
\end{equation}
\item For 
\[  S := \min_{r\in\mathcal{C}} \min_{x\in\partial M_r} d(x, x^r)\, , \]
let $0<a < 2S - S_0$
and $S_{jk} < S_0 + a$ and 
for all $\delta > 0$
\begin{equation}\label{mualphamu}
|\mu_j - \mu_k |= \expord{(a-\delta)} \; .
\end{equation}
We define the closed ``ellipse'' 
\begin{equation}\label{ellipse}
 E := \{ x\in \R^d\,|\, d^j(x)+ d^k(x) \leq S_0 + a \}
\end{equation}
and assume that $E\subset \stackrel{\circ}{M}_{j}\cup
\stackrel{\circ}{M}_{k}$.
\item For $R>0$ we set
\begin{equation}\label{HR}
\Hd_{d,R}:= \{ x\in \R^d\, |\, -R<x_d<0\} 
\end{equation}
and choose $R>0$ large enough such that 
\begin{equation}\label{EHR}
E\cap \{x\in\R^d\,|\, x_d\leq -R\} = \emptyset\, , \qquad E\cap \Hd_{d,R} \subset \stackrel{\circ}{M}_j 
\qquad \text{and}\qquad E\cap \Hd_{d,R}^c \subset   \stackrel{\circ}{M}_k\, .
\end{equation}
\een
\end{hyp}

The tunneling between the wells $x^j$ and $x^k$ can be described by the interaction term 
\begin{equation}\label{wjkdef}
w_{jk} = \skpd{v_j}{\bigl(\id - \id_{M_k}\bigr) T_\ep v_k} = \skpd{v_j}{\bigl[T_\ep, \id_{M_k}\bigr] v_k}
\end{equation}
introduced in \cite{kleinro4}, Theorem 1.5 (cf. Theorem \ref{ealphafalpha}).

The main topic of this paper is to derive complete asymptotic expansions for $w_{jk}$, using 
the approximate eigenfunctions we constructed in \cite{kleinro5}.

\begin{rem}\label{wjkalt}
\ben
\item
Since the set $\Hd_{d,R}$ fulfills the assumptions on the set $\Omega$ introduced in \cite{kleinro4}, it follows from 
\cite{kleinro4}, Proposition 1.7, that the
interaction $w_{jk}$ between the two wells $x^j$ and $x^k$ (cf. Theorem  \ref{ealphafalpha}) is given by
\begin{equation}\label{wjkaltglg}
w_{jk} = \skpd{ [T_\ep, \id_{\Hd_{d,R}}]\id_{E}
v_j}{\id_{E} v_k} +
\expord{S_0 + a -\eta}  \, , \qquad \eta>0\; .
\end{equation}
In order to use symbolic calculus to compute asymptotic expansions of $w_{jk}$, we will 
smooth the characteristic function $\id_{\Hd_{d,R}}$ by convolution with a Gaussian.
\item It follows from the results in \cite{kleinro2} that, by Hypothesis \ref{hypIMj},(3), the Dirichlet eigenvalues
$\mu_j$ and $\mu_k$ lie in $\ep^{\frac{3}{2}}$-intervals around some eigenvalues of the associated harmonic oscillators at the wells 
$x^j$ and $x^k$ as constructed in \cite{kleinro5}, (1.19). Thus we can use the approximate eigenfunctions and the weighted
estimates given in \cite{kleinro5}, Theorem 1.7 and 1.8 respectively.
\een
\end{rem}

Next we give assumptions on the geometric setting, more precisely on the geodesics
between the two wells given in Hypothesis \ref{hypkj}. First we consider the generic setting, where
there is exactly one minimal geodesic between the two wells.
Later on, we consider the more general situation where the minimal geodesics build a manifold. 

We recall from \cite{kleinro} that, as usual, geodesics are the critical points of the length functional of the Finsler structure induced by $\tilde{h}_0$. 

\begin{hyp}\label{hypgeo1} 
There is a unique minimal geodesic $\gamma_{jk}$ (with respect to the Finsler distance $d$) between 
the wells $x^j$ and $x^k$. Moreover, $\gamma_{jk}$ intersects the hyperplane $\Hd_d$
transversally at some point $y_0 = (y'_0, 0)$ (possibly after redefining the origin) and is nondegenerate at $y_0$ in the sense that, 
transversally to $\gamma_{jk}$, the function $d^k + d^j$ 
changes quadratically, i.e., the restriction of $d^j(x) + d^k(x)$ to $\Hd_d$ has a positive Hessian at $y_0$.
\end{hyp}

\begin{figure}[htbp]
\centering
\includegraphics[width=15cm]{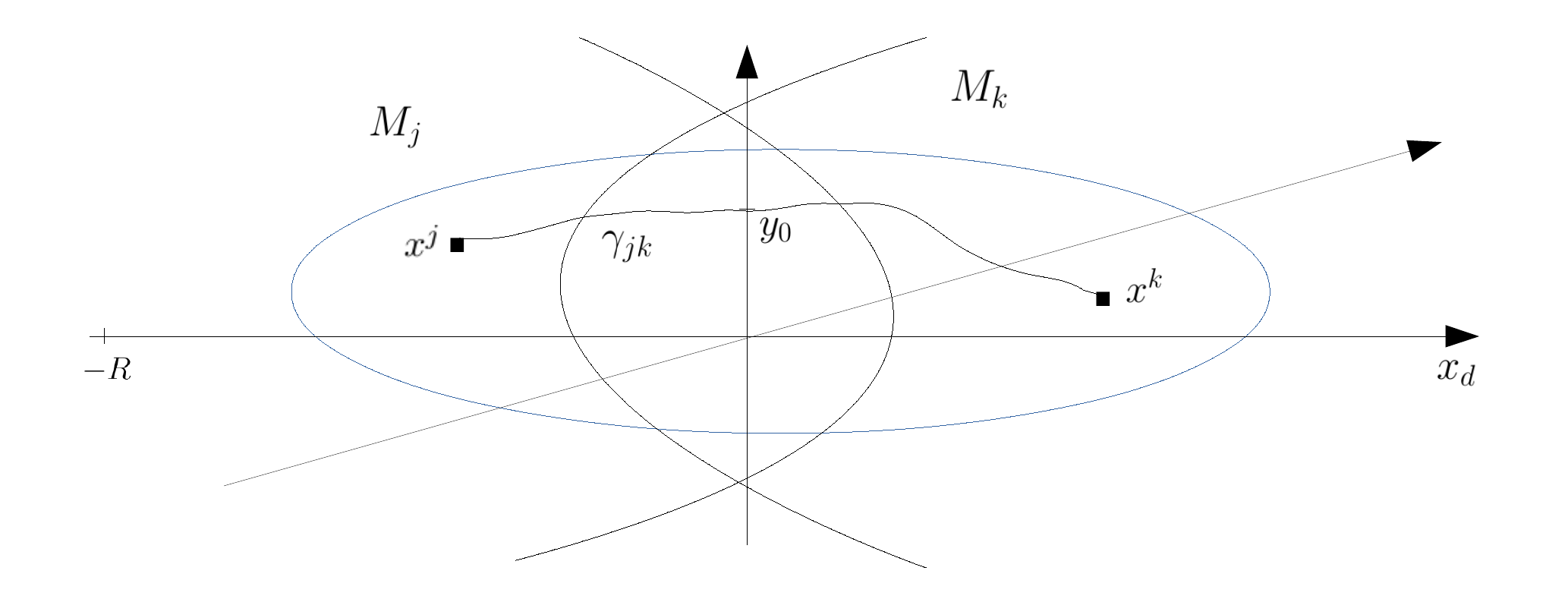}
\caption{The regions $E$, $M_j$ and $M_k$, the point $y_0$ and the curve $\gamma_{jk}$}
\label{Bild2}
\end{figure}

\begin{theo}\label{wjk-expansion}
Let $H_\ep$ be a Hamiltonian as in \eqref{Hepein} satisfying Hypotheses \ref{hyp1} and \ref{hyp2} and assume 
that Hypotheses \ref{hypIMj}, \ref{hypkj} and \ref{hypgeo1} are fulfilled.
For $m=j,k$, let $b^m\in\Ce_0^\infty(\R^d\times (0,\ep_0])$and $b^m_\ell\in \Ce^\infty_0 (\R^d), \, \ell\in \Z/2, \ell\geq -N_m$ 
for some $N_m\in \N$ be such
that the approximate eigenfunctions $\hat{v}_m^\ep\in\ell^2(\disk)$ of the
Dirichlet operators $H_\ep^{M_m}$ constructed in 
\cite{kleinro5}, Theorem 1.7, have asymptotic expansions
\begin{equation} \label{hatvm}
\hat{v}_m^\ep (x; \ep) = \ep^{\frac{d}{4}} e^{-\frac{d^m(x)}{\ep}} b^m(x; \ep)\quad\text{with}\quad 
b^m(x;\ep) \sim \sum_{\natop{\ell\in \Z/2}{\ell\geq -N_m}} \ep^\ell b^m_\ell \; .
\end{equation}
Then there is a sequence $(I_p)_{p\in \N/2}$ in $\R$ such that
\[ 
w_{jk} \sim \ep^{\frac{1}{2}-(N_j + N_k)} e^{-S_{jk}/\ep} \sum_{p\in \N/2} \ep^p I_p \, . 
\]
The leading order is given by
\begin{equation}\label{0thm1} 
I_0 =  \frac{(2\pi)^{\frac{d-1}{2}}}{\sqrt{\det D^{2}_\perp(d^j + d^k) (y_0)}} b^k_{-N_k}(y_0) 
 \sum_{\eta\in\Z^d} \tilde{a}_\eta(y_0) 
\eta_d \sinh \bigl(\eta\cdot \nabla d^j (y_0)\bigr) b^j_{-N_j}(y_0) 
\end{equation}
where we set $\tilde{a}_{\frac{\gamma}{\ep}}(x) := a_\gamma^{(0)}(x)$ and 
\begin{equation}\label{0athm1} 
D^{2}_\perp f := \Bigl( \partial_r\partial_p f \Bigr)_{1\leq r,p\leq d-1}\; . 
\end{equation}
\end{theo}

\begin{rem}
\ben
\item
The sum on the right hand side of \eqref{0thm1} is equal to the leading order of $\frac{1}{i} \Op_\ep^{\T}\bigl(w\bigr) \Psi b^j(y_0)$ where
\begin{equation}\label{current} 
w(x,\xi) := \partial_{\xi_d} t_0 (x, \xi - i \nabla d^j(x)) = 
-i \sum_{\gamma\in\disk} a_\gamma^{(0)}(x) \frac{\gamma_d}{\ep} e^{-\frac{i}{\ep}\gamma\cdot (\xi - i\nabla d^j (x))}\; . 
\end{equation}
To interpret this term (and formula \eqref{0thm1}) semiclassically, observe that $v(x,\xi) := \partial_\xi t_0 (x,\xi)$
is - by Hamilton's equation - the velocity field associated to the leading order kinetic Hamiltonian $t_0$ (or Hamiltonian $h_0 = t_0 + V_0$), 
evaluated on the physical phase space $T^*\R^d$. In \eqref{current}, with respect to the momentum variable, 
the phase space is pushed into the complex domain, over the
region $M_j\subset \R^d$ from Hypothesis \ref{hypIMj}
\[ T^*M_j \ni (x, \xi) \mapsto (x, \xi - i \nabla d^j(x)) \in \Lambda \subset T^*M_j\otimes \C\subset \C^{2d}\; .          \]
The smooth manifold $\Lambda$ lies as a graph over $T^*M_j$ and projects diffeomorphically. In some sense the complex deformation $\Lambda$
structurally stays as close a possible to the physical phase space $T^*M_j$, being both $\R$-symplectic and $I$-Langrangian.

We recall the basic definitions (see \cite{sjo} or \cite{hesjo3}): The standard symplectic form in $\C^{2d}$ is $\sigma = \sum_j d\zeta_j \wedge dz_j$ where
$z_j = x_j + i y_j$ and $\zeta_j = \xi_j + i \eta_j$. It decomposes into 
\begin{align*} 
\Re \sigma &= \frac{1}{2}(\sigma + \bar{\sigma}) = \sum_j d\xi_j \wedge dx_j - d\eta_j \wedge dy_j\\
\Im \sigma &= \frac{1}{2}(\sigma - \bar{\sigma}) = \sum_j d\xi_j \wedge dy_j + d\eta_j \wedge dx_j\, .
\end{align*}
Both $\Re \sigma$ and $\Im \sigma$ are real symplectic forms in $\C^{2d}$, considered as a real space of dimension $4d$.
A submanifold $\Lambda$ of $\C^{2d}$ (of real dimension $2d$) is called $I$-Langrangian if it is Lagrangian for $\Im \sigma$, and
$\Lambda$ is called $\R$-symplectic if $\Re\sigma|_\Lambda$ - which denotes the pull back under the embedding 
$\Lambda \hookrightarrow \C^{2d}$ - is non-degenerate. In our example, one checks in a straightforward way that both 
$T^*M_j$ and $\Lambda$ are $\R$-symplectic and $I$-Langrangian. In this paper we shall not explicitly use this structure of $\Lambda$ 
(it is essential for the microlocal theory of resonances, see \cite{hesjo3}); rather, the manifold $\Lambda$ appears somewhat mysteriously through explicit 
calculation.

Still, it seems to be physical folklore that both tunneling and resonance phenomena are related to complex deformations of 
phase space. Our formulae make this precise in the following sense: The leading order $I_0$ of the tunneling is given by 
the velocity field $v|_\Lambda$ (in the direction $e_d$) where $\Lambda$ is the $\R$-symplectic, $I$-Langrangian manifold obtained as
deformation of $T^*M_j$ through the field $\nabla d^j$ induced by the Finsler distance $d^j$, the leading amplitudes $b^j_{-N_j}(y_0)$,
$b^k_{-N_k}(y_0)$ of the WKB expansions and the ``hydrodynamical factor'' $\sqrt{\det D^{2}_\perp(d^j + d^k) (y_0)}$ describing 
deviations from the shortest path connecting the two potential minima.

Thus, in some sense, tunneling is described by a matrix element of a current (at least in leading order). On physical grounds it is 
perhaps very plausible that such formulae should hold in the semiclassical limit in any case which exhibits a leading order
Hamiltonian. That this is actually true in the case of difference operators considered in this article is
conceptually a main result of this paper. For pseudodifferential operators in $\R^d$ this is proven in \cite{hepar}.
For a less precise, but conceptually related, statement see \cite{kleinro4}.
\item If $\mu_j$ and $\mu_k$ correspond to the ground state energy of the harmonic oscillators associated to the Dirichlet 
operators at the wells (see \cite{kleinro5}), we have $N_j = N_k = 0$.
Moreover $b^j(y_0)$ and $b^k(y_0)$ are strictly positive. Thus if $\gamma_{jk}$ intersects $\Hd_d$ orthogonal, it follows from
Hypothesis \ref{hyp1}, (1)(ii), that
$I_0>0$.

\een
\end{rem}

 If there are finitely many geodesics connecting $x^j$ and $x^k$, separated away from the endpoints, their contributions to the 
interaction $w_{jk}$ simply add up (as conductances working in parallel do). 
This is more complicated (but conceptually similar) in the case where the minimal geodesics form a manifold.

\begin{hyp}\label{hypgeo2} 
For some $1\leq \ell < d$, the minimal geodesics from $x^j$ to $x^k$ (with respect to the Finsler distance d) form an
orientable $\ell+1$-dimensional submanifold $G$ of $\R^d$ (possibly singular at $x^j$ and $x^k$). Moreover $G$ intersects the hyperplane 
$\Hd_d$
transversally (possibly after redefining the origin). Then  
\begin{equation}\label{Gnull}
 G_0:= G \cap \Hd_d
\end{equation}
is a $\ell$-dimensional submanifold of $G$. 
\end{hyp}

We shall show in Step 2 of the proof of Theorem \ref{wjk-expansion2} below (assuming only Hypothesis \ref{hypgeo2}) that any system of linear
independent normal vector fields $N_m,\, m=\ell+1, \ldots , d$, on $G_0$ possesses an extension to a suitable tubular neighborhood of $G_0$
as a family of commuting vector fields. In particular, with such a choice of vector fields $N_m,\, m=\ell + 1, \ldots d$, 
\begin{equation}\label{transversHessian} 
D^2_{\perp, G_0}\bigl(d^j + d^k\bigr) := \Bigl( N_m N_n (d^j + d^k)|_{G_0}\Bigr)_{\ell+1\leq m,n \leq d-1} 
\end{equation}
is a symmetric matrix. We assume

\begin{hyp}\label{hypgeo2a}
The transverse Hessian $ D^2_{\perp, G_0}\bigl(d^j + d^k\bigr)$ of $d^j + d^k$ at $G_0$ defined in \eqref{transversHessian} is positive for all points
on $G$ (which we shortly denote as $G$ being non-degenerate at $G_0$).
\end{hyp}

\begin{theo}\label{wjk-expansion2}
Let $H_\ep$ be a Hamiltonian as in \eqref{Hepein} satisfying Hypotheses \ref{hyp1} and \ref{hyp2} and assume 
that Hypotheses \ref{hypIMj}, \ref{hypkj}, \ref{hypgeo2} and \ref{hypgeo2a} are fulfilled.
For $m=j,k$, let $\hat{v}_m^\ep$ be as in \eqref{hatvm}.
Then there is a sequence $(I_p)_{p\in \N/2}$ in $\R$ such that
\[ w_{jk} \sim \ep^{-(N_j + N_k)} \ep^{(1-\ell)/2} e^{-S_{jk}/\ep} \sum_{p\in \N/2} \ep^p I_p \, . \]
The leading order is given by
\begin{equation}\label{0thm2}
 I_0 = (2\pi)^{(d-(\ell+1))/2} \int_{G_0} \frac{1}{\sqrt{\det D^2_{\perp, G_0}\bigl(d^j + d^k\bigr)(y)}} b^k(y) b^j(y) \sum_{\eta \in\Gamma} 
\tilde{a}_\eta(y) \eta_d \sinh \bigl(\eta\cdot \nabla d^j (y)\bigr) \, d\sigma (y)
\end{equation}
where we used the notation given in Theorem \ref{wjk-expansion}.
\end{theo}

We remark that - after appropriate complex deformations - an essential idea in the proof of Theorem \ref{wjk-expansion} and
Theorem \ref{wjk-expansion2} is to replace discrete sums by integrals up to a very small error and then apply stationary phase.
This replacement of a sum by an integral is considerably more involved in the case of Theorem \ref{wjk-expansion2} and
represents  a main difficulty in the proof. \\

Concerning the case of the Schr\"odinger operator, results analog to Theorem \ref{wjk-expansion2} certainly hold true, but to the
best of our knowledge are not published (for the somewhat related case of resonances, see \cite{hesjo3}).\\

The outline of the paper is as follows.

Section \ref{section2} consists of preliminary results needed for the proofs of both theorems. The proofs of Theorem \ref{wjk-expansion} 
and Theorem \ref{wjk-expansion2}, 
are then given in 
in Section \ref{section3} and Section \ref{section4} respectively. In Section \ref{section5} we give some additional results on the
interaction matrix. Appendix \ref{app1} consists of some results for the symbolic calculus of periodic symbols. In Appendix \ref{app2} we 
recall a basic result from \cite{kleinro4} about the tunneling where the interaction matrix $w_{jk}$ is defined.\\

{\sl Acknowledgements.} The authors thank B. Helffer for many valuable discussions and remarks on the
subject of this paper.

\section{Preliminary Results on the interaction term $w_{jk}$}\label{section2}

Throughout this section we assume that Hypotheses \ref{hyp1}, \ref{hyp2} and \ref{hypIMj} are fulfilled and 
the interaction term $w_{jk}$ is as defined in \eqref{wjkdef}. \\

Following \cite{hesjo2} and \cite{hepar}, we set for some $C_0>0$ 
\begin{equation}\label{phinull}
\phi_0(t) := iC_0 t^2\qquad\text{and}\qquad \phi_s(t) := \phi_0(t-s)\, , \qquad s,t\in \R
\end{equation}
and define the multiplication operator 
\begin{equation}\label{pis}
\pi_s(x) := \frac{\sqrt{C_0}}{\sqrt{\pi \ep}} e^{\frac{i}{\ep} \phi_s(x_d)} = \frac{\sqrt{C_0}}{\sqrt{\pi \ep}} e^{-\frac{C_0}{\ep} (x_d-s)^2}
\, , \quad x\in\R^d
\end{equation} 
where the factor is chosen such that $\int_\R \pi_s\, ds = 1$.

\begin{prop}\label{prop1}
\begin{equation}\label{0prop1} 
w_{jk} = \int_{-R}^0 \skpd{\bigl[T_\ep, \pi_s\bigr] \id_E v_j}{\id_E v_k} \, ds+ \expord{S_0 + a - \eta}\, , \qquad \eta > 0 \; . 
\end{equation}
\end{prop}

\begin{proof}
By \cite{kleinro4}, Proposition 4.2, we get by arguments similar to those given in the proof of \cite{kleinro4}, Theorem 1.7, for all $\eta>0$
\[
w_{jk} = \skpd{\id_E v_j}{ \bigl[T_\ep, \id_{M_k}\bigr] \id_E v_k} + \expord{S_0 + a - \eta} \; .\]
Using $\int_\R \pi_s\, ds = 1$ this yields
\begin{equation}\label{17prop1}
w_{jk}=\skpd{\int_{-R}^0 \pi_s\, ds \,\id_E v_j}{ \bigl[T_\ep, \id_{M_k}\bigr] \id_E v_k}  + A + B + \expord{S_0 + a - \eta}
\end{equation}
where
\begin{align*}
A &:= \skpd{\int_{-\infty}^{-R} \pi_s\, ds \,\id_E v_j}{ \bigl[T_\ep, \id_{M_k}\bigr] \id_E v_k}\quad\text{and} \\
B &:= \skpd{\int_0^\infty \pi_s\, ds\, \id_E v_j}{ \bigl[T_\ep, \id_{M_k}\bigr] \id_E v_k}  \; .
\end{align*}
By the assumptions on $E$ and $R$ in Hypothesis \ref{hypkj}, we have $A=0$. 
In order to show that $B=\expord{S_0 + a - \eta}$, we use \cite{kleinro4}, Lemma 5.1, telling us that for all $C>0$ and $\delta>0$
\begin{equation}\label{2prop1} 
\bigl[T_\ep, \id_{M_k}\bigr] = \id_{\delta M_k}\bigl[ T_\ep, \id_{M_k}\bigr] \id_{\delta M_k} + \expord{C}
\end{equation}
where, for any $A\subset \R^d$, we set
\begin{equation}\label{deltaA}
\delta A := \{ x\in \R^d\, |\, \dist (x, \partial A) \leq \delta\}\; .
\end{equation}
Setting
\begin{equation}\label{12prop1} 
b_{\delta,k}:= \min \{ |x_d|\, |\, x\in E\cap \delta M_k\}  \; ,
\end{equation}
we write
\begin{equation}\label{12aprop1}
 \int_0^\infty \pi_s \, ds \,\id_{E\cap \delta M_k}(x) = 
e^{-\frac{C_0}{\ep} b_{\delta,k}^2} \frac{\sqrt{C_0}}{\sqrt{\pi\ep}} \int_0^\infty \id_{E\cap \delta M_k}(x)
 e^{-\frac{C_0}{\ep}((x_d-s)^2 - b_{\delta,k}^2)} \, ds\; .
\end{equation}
Since $\Hd^c_{d,R}\cap E \subset \stackrel{\circ}{M}_k$ by Hypothesis \ref{hypkj}, it follows that, for $\delta>0$ sufficiently small, $x_d<0$ for 
$x\in E\cap \delta M_k$ and thus 
$|x_d - s| \geq |x_d| \geq b_{\delta,k}>0$ for $s\geq 0$.
Therefore the substitution $z^2 = \frac{C_0}{\ep} \bigl((x_d-s)^2 - b_{\delta,k}^2\bigr)$ on the right hand side of \eqref{12aprop1} yields 
$\frac{1}{\sqrt{\ep}}ds \leq 
\frac{z\sqrt{\ep}}{b_{\delta,k}}dz$ and thus by straightforward 
calculation for some $C_\delta>0$
\begin{equation}\label{1prop1}
\sup_x \Bigl| \int_0^\infty \pi_s \, ds \id_{E\cap \delta M_k}(x)\Bigr| \leq C_\delta \sqrt{\ep} e^{-\frac{C_0}{\ep}b_{\delta,k}^2}\; .
\end{equation}
Combining \eqref{2prop1} and \eqref{1prop1} and using $d^j(x) + d^k(x) \geq S_{jk}$ gives for all $\delta>0$
\begin{equation}\label{3prop1}
|B| \leq C_\delta \sqrt{\ep} e^{-\frac{C_0}{\ep}b_{\delta,k}^2} e^{-\frac{S_{jk}}{\ep}} \Bigl\|e^{\frac{d^j}{\ep}} v_j\Bigr\|_{\ell^2} 
\Bigl\| e^{\frac{d^k}{\ep}} [T_\ep, \id_{\delta M_k}\id_{M_k}]\id_{\delta M_k} v_k\Bigr\|_{\ell^2}\, .
\end{equation}
The definition of $T_\ep$ and $\id_{M_k}v_k = v_k$ yield $\bigl[T_\ep, \id_{M_k}\bigr]v_k(x) = \bigl(\id - \id_{M_k}\bigr)(x)\sum_\gamma a_\gamma (x;\ep) 
v_k(x+\gamma)$. 
The triangle inequality $d^k(x) \leq d(x, x+\gamma) + d^k(x+\gamma)$ 
and the Cauchy-Schwarz-inequality with respect to $\gamma$ therefore give
\begin{multline}\label{4prop1}
\Bigl\| e^{\frac{d^k}{\ep}} \id_{\delta M_k} [T_\ep, \id_{M_k}]\id_{\delta M_k} v_k\Bigr\|^2_{\ell^2} = 
\sum_{x\in\disk}\Bigl|\id_{\delta M_k}\bigl(\id - \id_{M_k}\bigr)(x) \sum_{\gamma\in\disk} a_\gamma(x;\ep) e^{\frac{d^k(x)}{\ep}} 
\bigl(\id_{\delta M_k} v_k\bigr)(x+\gamma)\Bigr|^2 \\
\leq \sum_{x\in M_{k,\ep}^c\cap \delta M_k} \Bigl( \sum_{\gamma \in\disk} \bigl| a_\gamma (x;\ep) e^{\frac{d(x, x+\gamma)}{\ep}}
\langle \gamma\rangle_\ep^{\frac{d+\eta}{2}}\bigr|^2 \Bigr)
\Bigl( \sum_{\gamma\in\disk} \bigl| e^{\frac{d^k(x+\gamma)}{\ep}} \bigl(\id_{\delta M_k} v_k\bigr)(x+\gamma) 
\langle \gamma\rangle_\ep^{-\frac{d+\eta}{2}}\bigr|^2 \Bigr)
\end{multline}
where we set $\langle \gamma\rangle_\ep := \sqrt{\ep^2 + |\gamma|^2}$.
By Hypothesis \ref{hypIMj}, for $\eta>0$ chosen consistently, the first factor on the right hand side of \eqref{4prop1} is 
bounded by some constant $C>0$ uniformly with respect to $x$.  
Changing the order of summation therefore yields
\begin{align}
\Bigl\| e^{\frac{d^k}{\ep}} \id_{\delta M_k}[T_\ep, \id_{M_k}]\id_{\delta M_k} v_k\Bigr\|^2_{\ell^2} &\leq 
C \sum_{\gamma\in \disk} \langle \gamma\rangle_\ep^{-(d+\eta)} 
\sum_{x\in M_{k,\ep}^c\cap \delta M_k}\bigl|e^{\frac{d^k(x+\gamma)}{\ep}}  \bigl(\id_{\delta M_k} v_k\bigr)(x+\gamma)\bigr|^2 \nonumber\\ 
&\leq \tilde{C} \Bigl\| e^{\frac{d^k}{\ep}}v_k\Bigr\|^2_{\ell^2}\label{5prop1}\; .
\end{align}
We now insert \eqref{5prop1} into \eqref{3prop1} and use that, by \cite{kleinro5}, Proposition 3.1, the 
Dirichlet eigenfunctions decay exponentially fast, i.e. there is a constant 
$N_0 \in \N$ such that
$\| e^{\frac{d^i}{\ep}}v_i\|_{\ell^2} \leq \ep^{-N_0}$  for $i=j, k$. 
This gives for any $\eta>0$
\[ |B| \leq C e^{-\frac{1}{\ep}(C_0 b_{\delta,k}^2 + S_{jk} - \eta)}\; . \]
Since $b_{\delta,k}>0$ we can choose $C_0$ such that $C_0 b_{\delta,k}^2 + S_{j,k} \geq S_0 + a$, showing that
$B = \expord{S_0 + a - \eta}$ for $C_0$ sufficiently large and therefore by \eqref{17prop1}
\begin{equation}\label{6prop1}
w_{jk} = \skpd{\int_{-R}^0 \pi_s\, ds \,\id_E v_j}{ \bigl[T_\ep, \id_{M_k}\bigr] \id_E v_k} + \expord{S_0 + a - \eta}\; .
\end{equation}
In order to get the stated result, we use the symmetry of $T_\ep$ to write
\begin{align}
&\skpd{\int_{-R}^0 \pi_s\, ds \,\id_E v_j}{ \bigl[T_\ep, \id_{M_k}\bigr] \id_E v_k}  \nonumber\\
&\qquad=\skpd{T_\ep \int_{-R}^0 \pi_s\, ds\, \id_E v_j}{\id_E v_k} - \skpd{\int_{-R}^0 \pi_s\, ds\, \id_E v_j}{ \id_{M_k} T_\ep \id_E v_k}\nonumber \\
&\qquad= \skpd{\bigl[ T_\ep, \int_{-R}^0 \pi_s\, ds\bigr] \id_E v_j}{ \id_E v_k} + \sum_{i=1}^5 R_i\label{7prop1}
\end{align}
where by commuting $T_\ep$ with $\id_E$ and inserting $\id_{M_j} + \id_{M_j^c}$ in $R_2$ and $R_3$
\begin{align*}
R_1 &:= \skpd{\int_{-R}^0 \pi_s\, ds \bigl[T_\ep, \id_E\bigr] v_j}{\id_E v_k}  \\
R_2 &:= \skpd{\int_{-R}^0 \pi_s\, ds \,\id_E\id_{M_j} T_\ep v_j}{ \id_E v_k} \\
R_3 &:= \skpd{\int_{-R}^0 \pi_s\, ds \,\id_E\id_{M_j^c} T_\ep v_j}{\id_E v_k}  \\
R_4 &:= -\skpd{\int_{-R}^0 \pi_s\, ds\, \id_E v_j}{ \id_E\id_{M_k} T_\ep v_k} \\
R_5 &:= -\skpd{\int_{-R}^0 \pi_s\, ds \,\id_E v_j}{ \id_{M_k}\bigl[T_\ep, \id_E\bigr] v_k}\; .
\end{align*}

We are now going to prove that $\bigl|\sum_i R_i\bigr| = \expord{S_0 + a - \eta}$ for all $\eta>0$.

Since $\id_E(x) \bigl(\id_E(x+\gamma) -\id_E(x)\bigr)$ is equal to $-1$ for $x\in E, x+\gamma \in E^c$ and zero otherwise, we have
\begin{align} 
\bigl|R_1\bigr| &= \Bigl| \sum_{x,\gamma\in \disk} \int_{-R}^0 \pi_s\, ds \,  v_k(x) a_\gamma(x; \ep) v_j(x+\gamma) \id_E (x)
\bigl( \id_E(x+\gamma) - \id_E (x)\bigr)\Bigr| \nonumber\\
&=  \Bigl| \sum_{x,\gamma\in \disk} \int_{-R}^0 \pi_s\, ds \,  v_k(x) a_\gamma(x; \ep) v_j(x+\gamma) \id_E (x) \id_{E^c}(x+\gamma)\Bigr|\; . \label{8prop1}
\end{align}
Using for the first step that $d^j(x+\gamma) + d^k(x+\gamma) \geq S_0 + a$ for $x+\gamma\in E^c$ and for the second 
step the triangle inequality for $d$, we get
\begin{align} 
\text{rhs} \eqref{8prop1} &\leq e^{-\frac{S_0 + a}{\ep}} \sum_{x,\gamma\in\disk} \Bigl| 
 \int_{-R}^0 \pi_s\, ds \,  e^{\frac{d^k(x+\gamma)}{\ep}}v_k(x) a_\gamma(x; \ep) e^{\frac{d^j(x+\gamma)}{\ep}}v_j(x+\gamma) \id_E (x) 
\id_{E^c}(x+\gamma)\Bigr|\nonumber \\
&\leq  e^{-\frac{S_0 + a}{\ep}} \sum_{x\in\disk} \Bigl| 
 \Bigl( \int_{-R}^0 \pi_s\, ds \,  e^{\frac{d^k}{\ep}}\id_E v_k\Bigr)(x)\Bigr|\sum_{\gamma\in \disk} \Bigl| a_\gamma(x; \ep) 
e^{\frac{d(x, x+\gamma)}{\ep}}
\Bigl( e^{\frac{d^j}{\ep}}\id_{E^c} v_j\Bigr) (x+\gamma)\Bigr| \nonumber\\
&\leq  e^{-\frac{S_0 + a}{\ep}} \Bigl\| e^{\frac{d^k}{\ep}}v_k\Bigr\|_{\ell^2} \Biggl(\sum_{x\in \disk}\Bigl| \sum_{\gamma\in \disk}  a_\gamma(x; \ep)
 e^{\frac{d(x, x+\gamma)}{\ep}}\Bigl( e^{\frac{d^j}{\ep}} v_j\Bigr) (x+\gamma)\Bigr|^2 \Biggr)^{1/2} \label{9prop1}
\end{align}
where in the last step we used the Cauchy-Schwarz-inequality with respect to $x$ and $\int_\R \pi_s\, ds = 1$. 
By Cauchy-Schwarz-inequality with respect to $\gamma$ analog to \eqref{4prop1} and \eqref{5prop1} we get 
\begin{align}
 &\sum_{x\in \disk}\Bigl| \sum_{\gamma\in \disk}  a_\gamma(x; \ep)
 e^{\frac{d(x, x+\gamma)}{\ep}}\Bigl( e^{\frac{d^j}{\ep}} v_j\Bigr) (x+\gamma)\Bigr|^2  \nonumber\\
&\quad = \sum_{x\in \disk} \Bigl( \sum_{\gamma \in\disk} \bigl| a_\gamma (x;\ep) e^{\frac{d(x, x+\gamma)}{\ep}}
\langle \gamma\rangle_\ep^{(d+\eta)/2}\bigr|^2 \Bigr)
\Bigl( \sum_{\gamma\in\disk} \bigl| e^{\frac{d^j(x+\gamma)}{\ep}} v_k\bigr)(x+\gamma) \langle \gamma\rangle_\ep^{-(d+\eta)/2}\bigr|^2 \Bigr)\nonumber\\
&\quad \leq C \Bigl\| e^{\frac{d^j}{\ep}}v_j\Bigr\|^2_{\ell^2}\label{18prop1}\; .
\end{align}
Inserting \eqref{18prop1} into \eqref{9prop1} gives by \eqref{8prop1} together with \cite{kleinro5}, Proposition 3.1, for any $\eta>0$
\begin{equation}\label{10prop1}
\bigl|R_1\bigr| \leq C  e^{-\frac{S_0 + a}{\ep}} \Bigl\| e^{\frac{d^k}{\ep}}v_k\Bigr\|_{\ell^2}  \Bigl\| e^{\frac{d^j}{\ep}}v_j\Bigr\|_{\ell^2}
\leq C  e^{-\frac{S_0 + a-\eta}{\ep}}\; .
\end{equation}
Analog arguments show 
\begin{equation}\label{16prop1}
|R_5| = \expord{S_0 + a - \eta}\; .
\end{equation}
We analyze $|R_2 + R_4|$ together, writing
\[
\bigl| R_2 + R_4 \bigr| \leq \sum_{x\in\disk} \int_{-R}^0 \pi_s\, ds \id_E(x) \Bigl| v_k(x) \bigl(\id_{M_j}T_\ep v_j\bigr)(x) - v_j(x) \bigl( \id_{M_k} 
T_\ep v_k\bigr) (x)\Bigr| \; .
\]
Now using that 
\[ v_k \id_{M_j}T_\ep v_j - v_j \id_{M_k}T_\ep v_k + V_\ep v_j v_k - V_\ep v_j v_k = v_k H_\ep^{M_j} v_j - v_j H_\ep^{M_k} v_k =
 (\mu_j - \mu_k) v_j v_k  \]
we get by Hypothesis \ref{hypkj}, Cauchy-Schwarz-inequality and since $d^j(x) + d^k(x) \geq S_{jk}$
\begin{align}
\bigl| R_2 + R_4 \bigr| &\leq |\mu_j - \mu_k| e^{-\frac{S_{jk}}{\ep}} \sum_{x\in\disk}  \int_{-R}^0 \pi_s\, ds \id_E(x) \Bigl| e^{\frac{d^j(x)}{\ep}} v_j(x) 
e^{\frac{d^k(x)}{\ep}} v_k(x)\Bigr|\nonumber\\
&\leq e^{-\frac{S_{jk} + a - \delta}{\ep}} \Bigl\| e^{\frac{d^j}{\ep}}v_j\Bigr\|_{\ell^2} \Bigl\| e^{\frac{d^k}{\ep}}v_k\Bigr\|_{\ell^2}\nonumber \\
&\leq C e^{-\frac{S_0 + a - \eta}{\ep}}\label{11prop1}
\end{align}
where in the last step we used again \cite{kleinro5}, Proposition 3.1, and $S_{jk}\geq S_0$.

The term $|R_3|$ can be estimated by methods similar to those used to estimate $|B|$ above. By Hypothesis \ref{hypkj} we have 
$E\cap \Hd_{d,R} \subset \stackrel{\circ}{M}_j$. Thus $x_d>0$ for $x\in E\cap M_j^c$ and, setting 
$b_j:= \min \{|x_d|\,|\, x\in E\cap M_j^c\}$, we have $|x_d - s| \geq |x_d|\geq b_j >0$ for $s\leq 0$. 
Thus we get analog to \eqref{12aprop1} and \eqref{1prop1}
\begin{equation}\label{13prop1}
\sup_x \Bigl| \int_{-R}^0 \pi_s \, ds \id_{E\cap M^c_j}(x)\Bigr| \leq C \sqrt{\ep} e^{-\frac{C_0}{\ep}b_j^2}
\end{equation}
and similar to \eqref{3prop1}, using Cauchy-Schwarz-inequality, 
\begin{equation}\label{14prop1}
\bigl| R_3\bigr| \leq  C \sqrt{\ep} e^{-\frac{1}{\ep}(C_0 b_j^2 + S_{jk})}\Bigl\| e^{\frac{d^k}{\ep}} v_k\Bigr\|_{\ell^2}
\Bigl\| e^{\frac{d^j}{\ep}} T_\ep v_j\Bigr\|_{\ell^2}
\end{equation} 
As in \eqref{4prop1} and\eqref{5prop1}, we estimate the last factor in \eqref{14prop1} as
\begin{multline*}
\bigl\| e^{\frac{d^j}{\ep}} T_\ep v_j\bigr\|^2_{\ell^2} = 
\sum_{x\in\disk}\Bigl|\sum_{\gamma\in\disk} a_\gamma(x;\ep) e^{\frac{d^j(x)}{\ep}} v_j\bigr)(x+\gamma)\Bigr|^2 \\
\leq \sum_{x\in \disk} \Bigl( \sum_{\gamma \in\disk} \bigl| a_\gamma (x;\ep) e^{\frac{d(x, x+\gamma)}{\ep}}
\langle \gamma\rangle_\ep^{\frac{d+\eta}{2}}\bigr|^2 \Bigr)
\Bigl( \sum_{\gamma\in\disk} \bigl| e^{\frac{d^j(x+\gamma)}{\ep}} v_j(x+\gamma) \langle \gamma\rangle_\ep^{-\frac{d+\eta}{2}}\bigr|^2 \Bigr)\\
\leq 
C \sum_{\gamma\in \disk} \langle \gamma\rangle_\ep^{-(d+\eta)} 
\sum_{x\in \disk}\bigl|e^{\frac{d^j(x+\gamma)}{\ep}}  v_j(x+\gamma)\bigr|^2 
\leq \tilde{C} \Bigl\| e^{\frac{d^j}{\ep}}v_j\Bigr\|^2_{\ell^2}\; .
\end{multline*}

Thus choosing $C_0$ such that
$C_0 b_j^2 + S_{jk} \geq S_0 + a$, we get again by \cite{kleinro5}, Proposition 3.1, for any $\eta>0$
\begin{equation}\label{15prop1}
\bigl| R_3\bigr| \leq  C e^{-\frac{1}{\ep}(C_0 b_j^2 + S_{jk})}\Bigl\| e^{\frac{d^k}{\ep}} v_k\Bigr\|_{\ell^2}
\Bigl\| e^{\frac{d^j}{\ep}} v_j\Bigr\|_{\ell^2} \leq C e^{-\frac{1}{\ep}(S_0 + a - \eta)}\; .
\end{equation}
Inserting \eqref{15prop1}, \eqref{11prop1}, \eqref{16prop1} and \eqref{10prop1} into \eqref{7prop1} yields \eqref{0prop1} 
by \eqref{6prop1} and interchanging of integration and summation. 

\end{proof}

In the next step we analyze the commutator in \eqref{0prop1} using symbolic calculus.

\begin{prop} \label{prop2}
For any $u\in \ell^2(\disk)$ compactly supported and $x\in \disk$ we have with the notation $\xi = (\xi', \xi_d)\in \T^d$
\begin{multline}\label{0prop2}
\bigl[ T_\ep, \pi_s\bigr] u(x) = \frac{\sqrt{C_0}}{\sqrt{\pi\ep}} (2\pi)^{-d} \sum_{y\in\disk} e^{\frac{i}{2\ep}( \phi_s(y_d) + \phi_s(x_d))} u(y)\\
\times 
\;\int_{[-\pi, \pi]^d} e^{\frac{i}{\ep}(y-x)\xi} \Bigl( t\bigl(x,  \xi', \xi_d - \frac{1}{2}\phi_s'(\frac{x_d + y_d}{2}); \ep \bigr) - 
t\bigl(x,  \xi', \xi_d + \frac{1}{2}\phi_s'(\frac{x_d + y_d}{2}); \ep\bigr) \Bigr) \, d\xi 
\end{multline}
where $\phi_s' (t) = \frac{d}{dt}\phi_s(t)= 2 i C_0 (t-s)$ and $T_\ep = \Op_\ep^\T (t)$ as given in \eqref{psdo2dTorus}.
\end{prop}

\begin{proof}
By Definition \ref{pseudo},(4), we have
\begin{align}\label{1prop2}
\bigl( T_\ep \pi_s u\bigr) (x) &=  \frac{\sqrt{C_0}}{\sqrt{\pi\ep}} (2\pi)^{-d} \sum_{y\in\disk} u(y)
\int_{[-\pi, \pi]^d} e^{\frac{i}{\ep}((y-x)\xi + \phi_s(y_d))} t(x,  \xi; \ep) \, d\xi \\
 \bigl(\pi_s T_\ep u\bigr) (x) &=  \frac{\sqrt{C_0}}{\sqrt{\pi\ep}} (2\pi)^{-d} \sum_{y\in\disk} u(y)
\int_{[-\pi, \pi]^d} e^{\frac{i}{\ep}((y-x)\xi + \phi_s(x_d))} t(x,  \xi; \ep) \, d\xi \label{2prop2}
\end{align}
Setting
\begin{equation}\label{3prop2}
\xi_{\pm} := \Bigl( \xi', \xi_d \pm \frac{1}{2} \phi_s'\bigl(\frac{x_d + y_d}{2}\bigr)\Bigr)
\end{equation}
we have
\begin{align}\label{4prop2}
(y-x)\xi + \phi_s(y_d) &= (y-x)\xi_+ + \frac{1}{2} \bigl( \phi_s(y_d) + \phi_s(x_d)\bigr) \\ 
(y-x)\xi + \phi_s(x_d) &= (y-x)\xi_- + \frac{1}{2} \bigl( \phi_s(y_d) + \phi_s(x_d)\bigr) \nonumber
\end{align}
In fact,
\begin{equation}\label{5prop2}
(y-x)\xi_{\pm} + \frac{1}{2}\bigl( \phi_s(y_d) + \phi_s(y_d)\bigr) = (y-x)\xi \pm (y_d - x_d) iC_0 
\Bigl( \frac{x_d + y_d}{2} - s\Bigr) + 
\frac{C_0 i}{2} \Bigl( (y_d - s)^2 + (x_d-s)^2\Bigr) \; .
\end{equation}
Writing $y_d - x_d = (y_d - s) - (x_d - s)$ and $\frac{x_d + y_d}{2} -s = \frac{1}{2} \bigl((x_d - s) + (y_d-s)\bigr)$ gives 
\begin{align*}
\text{rhs}\eqref{5prop2} &= (y-x)\xi \pm \frac{iC_0}{2} \Bigl( (y_d - s)^2 - (x_d-s)^2\Bigr) + \frac{iC_0}{2} 
\Bigl( (y_d-s)^2 + (x_d-s)^2\Bigr) \\
&= \begin{cases} (y-x)\xi + \phi_s(y_d)\; \text{ for }\; + \\ (y-x)\xi + \phi_s(x_d)\; \text{ for } \; - \end{cases}\; .
\end{align*}

Since, with respect to $\xi$, $t$ has an analytic continuation to $\C^d$, it is possible to combine the 
integrals in \eqref{1prop2} and \eqref{2prop2}
using the contour deformation given by the substitution \eqref{3prop2}.
To this end, we first need the following Lemma

\begin{Lem}\label{Lem1}
Let $f: \C\ra \C$ be analytic in $\Omega_b:= \{ z\in \C\,|\, \Im z < b\}$ for some $b>0$ and 
$2\pi$-periodic on the real axis, i.e. $f(x+2\pi) = f(x)$ for all $x\in \R$. 
Then for any $a<b$
\[ \int_{-\pi + ia}^{\pi + ia} f(z)\, dz = \int_{-\pi}^\pi f(x)\, dx\; . \]
\end{Lem}

\begin{proof}[Proof of Lemma \ref{Lem1}]
If $f$ is periodic on the real line, if follows that $f(z) = f(z+2\pi)$ for $z\in \O_b$ by the identity theorem. Then Cauchy's Theorem yields
\begin{equation}\label{1Lem1}
 \int_{-\pi + ia}^{\pi + ia} f(z)\, dz - \int_{-\pi }^{\pi } f(z)\, dz = \int_{-\pi + ia}^{-\pi} f(z)\, dz + \int_{\pi }^{\pi + ia} f(z)\, dz\; .
\end{equation}
The substitution $\tilde{z}= z - 2\pi$ in the last integral on the right hand side of \eqref{1Lem1} gives by the periodicity of $f$
\begin{align*}
\text{rhs}\eqref{1Lem1} &=  \int_{-\pi + ia}^{-\pi} f(z)\, dz + \int_{-\pi }^{-\pi + ia} f(\tilde{z} + 2\pi)\, d\tilde{z}\\
&=   \int_{-\pi + ia}^{-\pi} f(z)\, dz + \int_{-\pi }^{-\pi + ia} f(\tilde{z})\, d\tilde{z} 
= 0 \; ,
\end{align*}
proving the stated result.
\end{proof}

We come back to the proof of Proposition \ref{prop2}. For shortening the notation we set
\begin{equation}\label{6prop2}
 a:= \frac{1}{2} \phi_s'\bigl(\frac{x_d + y_d}{2}\bigr) = C_0 \Bigl( \frac{x_d + y_d}{2} - s\Bigr) \; .
\end{equation}
Inserting the substitution \eqref{3prop2} in \eqref{1prop2}, we get by \eqref{4prop2} and \eqref{6prop2}
\begin{align}
\bigl( T_\ep \pi_s u\bigr) (x) &=  \frac{\sqrt{C_0}}{\sqrt{\pi\ep}} (2\pi)^{-d} \sum_{y\in\disk} u(y)
\int_{[-\pi, \pi]^d} e^{\frac{i}{\ep}\bigl((y-x)\xi_+ + \frac{1}{2}(\phi_s(y_d) + \phi_s(x_d))\bigr)} t (x,  \xi; \ep) \, d\xi \nonumber\\
&=  \frac{\sqrt{C_0}}{\sqrt{\pi\ep}} (2\pi)^{-d} \sum_{y\in\disk} u(y)\int_{[-\pi, \pi]^{d-1}}\, d\xi'_+
  \nonumber \\
&   \hspace{5mm} \times\; \int_{-\pi-ia}^{\pi + i a} \, d(\xi_+)_d e^{\frac{i}{\ep}\bigl((y-x)\xi_+ + \frac{1}{2}(\phi_s(y_d) + \phi_s(x_d))\bigr)} 
t (x,  \xi_+', (\xi_+)_d - ia; \ep)\nonumber  \\
& =  \frac{\sqrt{C_0}}{\sqrt{\pi\ep}} (2\pi)^{-d} \sum_{y\in\disk} u(y)
\int_{[-\pi, \pi]^{d}} e^{\frac{i}{\ep}\bigl((y-x)\xi + \frac{1}{2}(\phi_s(y_d) + \phi_s(x_d))\bigr)} 
t (x,  \xi', \xi_d - ia; \ep) \, d\xi     \label{7prop2}
\end{align}
where in the last step we used Lemma \ref{Lem1}.

By analog arguments for \eqref{2prop2} we get
\begin{equation}\label{8prop2}
 \bigl(\pi_s T_\ep u\bigr) (x) =  \frac{\sqrt{C_0}}{\sqrt{\pi\ep}} (2\pi)^{-d} \sum_{y\in\disk} u(y)
\int_{[-\pi, \pi]^d} e^{\frac{i}{\ep}\bigl((y-x)\xi + \frac{1}{2}(\phi_s(y_d) + \phi_s(x_d))\bigr)} 
t (x,  \xi', \xi_d + ia; \ep) \, d\xi  
\end{equation}
and thus combining \eqref{7prop2} and \eqref{8prop2} gives \eqref{0prop2}.

\end{proof}

The idea is now to write the $s$-dependent terms in \eqref{0prop2} as 
$s$-derivative of some symbol.
To this end, we first introduce some smooth cut-off functions on the right hand side of \eqref{0prop1}. \\

Let $\chi_R\in \Ce_0^\infty (\R)$ be such that $\chi_R (s)=1$ for $s\in [-R, R]$ and $\chi_E\in \Ce_0^\infty (\R^d)$ such that
$\chi_E(x) = 1$ for $x\in E$.  Moreover we assume that $\chi_R(s) = \chi_R(-s)$ and $\chi_E(x) = \chi_E(-x)$. 
Then it follows directly from Proposition \ref{prop1} that
\begin{equation}\label{0prop1a}
 w_{jk} = \int_{-R}^0 \skpd{\chi_R(s) \bigl[T_\ep, \pi_s\bigr] \chi_E \id_E v_j}{\chi_E\id_E v_k}\, ds + 
\expord{S_0 + a - \eta}\, , \qquad \eta>0 \; . 
\end{equation}

\begin{prop}\label{prop3}
There are compactly supported smooth mappings 
\[ \R \ni s\mapsto q_s \in S_0^0(1)(\R^{2d}\times \T^d)\quad\text{and}\quad \R \ni s\mapsto r_s \in S_0^\infty(1)(\R^{2d}\times \T^d)\]
such that $q_s(x, y, \xi; \ep)$ and $r_s(x, y, \xi; \ep)$ 
have analytic continuations to $\C^d$ with respect to $\xi\in\R^d$ (identifying functions on $\T^d$ with periodic functions on $\R^d$).
Moreover, $q_s$ has an asymptotic expansion 
\begin{equation}\label{00prop3}
q_s(x, y, \xi; \ep) \sim \sum_{n=0}^\infty \ep^n q_{n,s}(x, y, \xi)\; .
\end{equation} 
and, setting $\sigma:= \frac{x_d + y_d}{2} - s$,
\begin{multline}\label{0prop3}
 \chi_R(s) \chi_E(x)\chi_E(y) e^{-\frac{C_0}{\ep}\sigma^2}\Bigl[ t \bigl(x, \xi', \xi_d - i C_0 
\sigma; \ep\bigr) -  t \bigl(x, \xi', \xi_d + i C_0\sigma; \ep\bigr)\Bigr]  \\
= \ep \partial_s \Bigl[  e^{-\frac{C_0}{\ep}\sigma^2} q_s(x, y, \xi; \ep)\Bigr]
 +  e^{-\frac{C_0}{\ep}\sigma^2} r_s (x, y, \xi; \ep)\; .
\end{multline}
\end{prop}

\begin{proof}
We first remark that by \eqref{talsexp}
\begin{multline}\label{1prop3}
 t \bigl(x, \xi', \xi_d - i C_0 \sigma; \ep \bigr) -  t \bigl(x, \xi', \xi_d + i C_0\sigma; \ep\bigr)
 \\ 
=\sum_{\gamma\in\disk} a_\gamma (x, \ep) e^{-\frac{i}{\ep}\gamma'\xi'} 
\Bigl[ e^{-\frac{i}{\ep}\gamma_d (\xi_d - iC_0\sigma)} - 
e^{-\frac{i}{\ep}\gamma_d(\xi_d + i C_0 \sigma)}\Bigr]\\
= \sum_{\gamma\in\disk} a_\gamma (x, \ep) e^{-\frac{i}{\ep}\gamma\xi} 
2 \sinh \Bigl( \frac{\gamma_d}{\ep}C_0 \sigma\Bigr)\; .
\end{multline}
Thus from the assumptions on $\chi_R$ and $\chi_E$ it follows that the left hand side of 
\eqref{0prop3} is odd with respect to $\sigma\mapsto -\sigma$. 
Modulo $S^\infty$, \eqref{0prop3} is equivalent to 
\begin{equation}\label{2prop3}
 \chi_R(s) \chi_E(x)\chi_E(y)\Bigl[ t \Bigl(x, \xi', \xi_d - i C_0 \sigma; \ep\Bigr) -  
t \Bigl(x, \xi', \xi_d + i C_0 \sigma; \ep\Bigr)\Bigr] 
= \bigl( 2 C_0 \sigma + \ep \partial_s\bigr) q_s(x, y, \xi; \ep)\; .
\end{equation}
Here $q$ is compactly supported in $x, y$ and $s$ (and thus in $\sigma$) and $q$ is even with respect to 
$\sigma\mapsto -\sigma$ since $\partial_s = -\partial_\sigma$.
We set
\begin{align}\label{3prop3}
  g_s(x, y, \xi; \ep) &:=   \chi_R(s) \chi_E(x)\chi_E(y)\frac{1}{2 C_0 \sigma} \Bigl(t \bigl(x, \xi', \xi_d - i C_0 \sigma; \ep \bigr) -  
t \bigl(x, \xi', \xi_d + i C_0 \sigma; \ep\bigr)\Bigr)\\
&= \sum_{\ell = 0}^\infty \ep^\ell g_{\ell, s}(x, y, \xi)\nonumber
\end{align}
where by \eqref{1prop3}
\begin{equation}\label{6prop3}
g_{\ell, s}(x, y, \xi) := 
 - \chi_R(s) \chi_E(x)\chi_E(y)\sum_{\gamma\in\disk} a^{(\ell)}_\gamma(x) 
e^{-\frac{i}{\ep}\gamma\xi} \frac{1}{C_0 \sigma} 
\sinh \Bigl( \frac{\gamma_d}{\ep} C_0 \sigma\Bigr)\; .
\end{equation}
Then \eqref{2prop3} can be written as
\begin{equation}\label{4prop3}
\Bigl( 1 +  \frac{\ep}{2 C_0 \sigma}\partial_s \Bigr) q_s(x, y, \xi; \ep) = g_s (x, y, \xi; \ep)\; .
\end{equation}
Formally \eqref{4prop3} leads to the von-Neumann-series
\begin{equation}\label{5prop3}
q_s (x, y, \xi; \ep) = \sum_{m=0}^\infty \ep^m \Bigl(-\frac{1}{2C_0 \sigma} \partial_s\Bigr)^m 
g_s (x, y, \xi; \ep)\; .
\end{equation}
Using \eqref{00prop3}, \eqref{3prop3} and Cauchy-product, \eqref{5prop3} gives
\begin{equation}\label{7prop3}
q_{n, s}(x, y, \xi) = \sum_{\ell+ m = n} \Bigl(-\frac{1}{2C_0 \sigma} \partial_s\Bigr)^m g_{\ell, s}(x, y, \xi)\; .
\end{equation}
By \eqref{3prop3} $g$ and $g_{\ell}$, $\ell\in\N$, are even with respect to $\sigma\mapsto -\sigma$. 
Moreover, the operator $\frac{1}{\sigma}\partial_s = -\frac{1}{\sigma}\partial_\sigma$ 
maps a monomial in $\sigma$ of order $2m$ to a 
monomial of order $\max \{0, 2m-2\}$. Thus, for $x,y\in\supp \chi_E$ and $s\in [-R, R]$, the right hand side of \eqref{7prop3} is 
well-defined and analytic and even in $\sigma$ for any $n\in\N$. 
In particular, it is bounded at $\sigma = 0$ or equivalently at $s=\frac{x_d + y_d}{2}$. 
Therefore $q_{n,s} \in S_0^0(1)(\R^{2d}\times \T^d)$ for any $n\in\N$ and it is $\Ce^\infty_0$ with respect to $s\in\R$.

By a Borel-procedure with respect to $\ep$ there exists a symbol 
$q_s\in S_0^0(1)(\R^{2d}\times \T^d)$ which is $\Ce_0^\infty$ as a function of $s\in\R$ such
that \eqref{00prop3} holds. Moreover, $\partial_s q_s(x, y, \xi; \ep)$ is analytic in $\xi$ by uniform convergence of the 
Borel procedure and the analyticity of $q_{n,s}$. Thus 
\eqref{0prop3} holds for some $r_s \in S_0^\infty(1)(\R^{2d}\times \T^d)$ and since the left hand side of \eqref{2prop3} 
has an analytic continuation to $\C^d$ with respect to $\xi$, the same is true for $r_s(x, y, \xi; \ep)$.

\end{proof}

We remark that by \eqref{7prop3} and \eqref{6prop3}, the leading order term $q_0$ at
the point $s=\frac{x_d + y_d}{2}$ is given by
\begin{align}
q_{0,\frac{x_d + y_d}{2}}(x, y, \xi) &= - \chi_R\Bigl(\frac{x_d + y_d}{2}\Bigr) \chi_E(x)\chi_E(y) 
\sum_{\gamma\in\disk} a_\gamma^{(0)}(x) \frac{\gamma_d}{\ep} e^{-\frac{i}{\ep}\gamma\xi} \nonumber \\
&= \frac{1}{i} \chi_E(y) \chi_E (x) \partial_{\xi_d} t_0 (x, \xi) \label{8prop3}
\end{align}
where in the second step we used \eqref{texpand} and the fact that $\chi_R(\frac{x_d + y_d}{2}) = 1$ for 
$x, y \in \supp \chi_E$.

We now define the operators $Q_s$ and $R_s$ on $\ell^2(\disk)$ by
\begin{align}\label{Q_sdef}
Q_s u(x) &:= \sqrt{\frac{C_0}{\ep \pi}} (2\pi)^{-d} \sum_{y\in\disk} e^{\frac{i}{2\ep} (\phi_s(y_d) + \phi_s(x_d))} u(y)
\int_{\T^d} e^{\frac{i}{\ep}(y-x)\xi} q_s(x, y, \xi; \ep) \, d\xi \\
R_s u(x) &:= \sqrt{\frac{C_0}{\ep \pi}} (2\pi)^{-d} \sum_{y\in\disk} e^{\frac{i}{2\ep} (\phi_s(y_d) + \phi_s(x_d))} u(y)
\int_{\T^d} e^{\frac{i}{\ep}(y-x)\xi} r_s(x, y, \xi; \ep) \, d\xi \label{R_sdef}
\end{align}

Then we get the following formula for the interaction term $w_{jk}$.

\begin{prop}\label{prop4}
  For $Q_s$ given in \eqref{Q_sdef}, the interaction term is given by
\begin{equation}\label{0prop4} 
w_{jk} = \ep \skpd{Q_0 \id_E v_j}{ \id_E v_k} + O\Bigl(\ep^\infty e^{-\frac{1}{\ep}S_{jk}}\Bigr)\; .
\end{equation}
\end{prop}

\begin{proof}
We first remark that by the definition \eqref{phinull} of $\phi_s$ we have
\begin{equation}\label{1prop4}
\frac{i}{2\ep}\bigl( \phi_s(y_d) + \phi_s(x_d)\bigr) = -\frac{C_0}{\ep} \Bigl[ \bigl(\frac{x_d + y_d}{2}-s\bigr)^2 + \frac{1}{4} (y_d - x_d)^2\Bigr]\; .
\end{equation}
Combining Proposition \ref{prop2} with Proposition \ref{prop3} and \eqref{1prop4} gives
\begin{align}
\chi_R(s)\chi_E &[T_\ep, \pi_s] \chi_E \id_E v_j(x) = \sqrt{\frac{C_0}{\ep \pi}} (2\pi)^{-d} \sum_{y\in\disk} \id_E(y) v_j(y) 
e^{-\frac{C_0}{4\ep}(y_d - x_d)^2}\nonumber\\
 &\times\,\int_{\T^d} e^{\frac{i}{\ep}(y-x)\xi} \ep \partial_s 
\Bigl( e^{-\frac{C_0}{\ep}(\frac{x_d + y_d}{2} - s)^2} q_s(x, y, \xi; \ep)\Bigr)
 +  e^{-\frac{C_0}{\ep}(\frac{x_d + y_d}{2} - s)^2} 
 r_s (x, y, \xi; \ep)\, d\xi\nonumber\\
 &= \bigl( \ep\partial_s Q_s + R_s \bigr) \id_E v_j(x) \label{2prop4}
\end{align}
where the second equation follows from the definitions \eqref{Q_sdef} and \eqref{R_sdef}.
Thus by \eqref{0prop1a} we get for any $\eta>0$
\begin{align}\label{3prop4}
  w_{jk} &= \int_{-R}^0 \skpd{\bigl(\ep \partial_s Q_s + R_s\bigr) \id_E v_j}{\id_E v_k} \, ds + O\Bigl(e^{-\frac{S_0 + a - \eta}{\ep}}\Bigr)\\
  &= \ep \skpd{Q_0 \id_E v_j}{\id_E v_k} - S_1 + S_2 + O\Bigl(e^{-\frac{S_0 + a - \eta}{\ep}}\Bigr) \; ,\nonumber
\end{align}
where
\begin{align}\label{4prop4}
 S_1 &:= \ep \skpd{Q_{-R}\id_E v_j}{\id_E v_k}\\
 S_2 &:= \ep \int_{-R}^0 \skpd{R_s \id_E v_j}{\id_E v_k}\, ds\label{5prop4}
\end{align}
To analyse $S_2$, we first introduce the following notation, which will be used again later on. We set (see Definition \ref{pseudo})
\begin{align}\label{8prop4}
\tilde{u}_s(x) &:= e^{\frac{i}{2\ep}\phi_s(x_d)} u(x) = e^{-\frac{C_0}{2\ep}(x_d - s)^2} u(x) \\
\tilde{Q}_s &:= \widetilde{\Op}_\ep^\T \Bigl(\sqrt{\frac{C_0 \ep}{\pi}} q_s\Bigr) \label{9prop4}\\
\tilde{R}_s &:=  \widetilde{\Op}_\ep^\T \Bigl(\sqrt{\frac{C_0 \ep}{\pi}} r_s\Bigr)\,, \label{11prop4}
\end{align}
then 
\begin{equation}\label{11aprop4}
 \ep \skpd{Q_{s}u}{v} = \skpd{\tilde{Q}_s \tilde{u}_s}{\tilde{v}_{s}} \quad\text{ and }\quad 
 \ep \skpd{R_{s}u}{v} = \skpd{\tilde{R}_s \tilde{u}_s}{\tilde{v}_{s}}\; . 
\end{equation}
To analyse $S_2$ we write, using \eqref{11aprop4} 
\begin{align}\label{6prop4}
  \bigl| S_2\bigr| &=  \Bigl| \int_{-R}^0 \skpd{e^{-\frac{d^k}{\ep}} \tilde{R}_s e^{-\frac{d^k}{\ep}} e^{-\frac{(d^k + d^j)}{\ep}} 
	e^{\frac{d^j}{\ep}}\id_E \tilde{v}_{j,s}}
    {e^{\frac{d^k}{\ep}}\id_E \tilde{v}_{k,s}}\, ds \Bigr| \\
    &\leq e^{-\frac{S_{jk}}{\ep}}\int_{-R}^0 \Bigl\| e^{-\frac{d^k}{\ep}} \tilde{R}_s e^{-\frac{d^k}{\ep}}
e^{\frac{d^j}{\ep}}\id_E \tilde{v}_{j,s}\Bigr\|_{\ell^2}
 \Bigl\| e^{\frac{d^k}{\ep}}\id_E \tilde{v}_{k,s}\Bigr\|_{\ell^2}\, ds\; .\nonumber
\end{align}
Since $r_s\in S_0^\infty (1)(\R^{2d}\times \T^d)$, it follows from Corollary \ref{cor1app} together with Proposition \ref{prop3app} that for some $C>0$
\begin{align}
\bigl| S_2\bigr| &\leq C \ep^\infty e^{-\frac{S_{jk}}{\ep}}\int_{-R}^0 \Bigl\|e^{\frac{d^j}{\ep}}\id_E \tilde{v}_{j,s}
\Bigr\|_{\ell^2} \Bigl\| e^{\frac{d^k}{\ep}}\id_E \tilde{v}_{k,s}\Bigr\|_{\ell^2}\, ds\nonumber\\
&= O\Bigl(\ep^\infty e^{-\frac{S_{jk}}{\ep}}\Bigr)\label{7prop4}
\end{align}
where for the second step we used weighted estimates for the Dirichlet eigenfunctions given in \cite{kleinro5}, Proposition 3.1, together with 
the fact that $|\tilde{u}_s(x)|\leq |u(x)|$.

By \eqref{4prop4} and \eqref{11aprop4} we get
\begin{equation}
\bigl| S_1\bigr|  =  \Bigl|\skpd{\tilde{Q}_{-R}\id_E \tilde{v}_{j, -R}}{\id_E \tilde{v}_{k, -R}}\Bigr| 
\leq \bigl\| \id_E \tilde{Q}_{-R}\id_E \tilde{v}_{j, -R} \bigr\|_{\ell^2}\, \bigl\|\id_E \tilde{v}_{k, -R}\bigr\|_{\ell^2}\; .
\end{equation}
Again by Corollary \ref{cor1app} together with \eqref{8prop4}, \eqref{9prop4} and since $q_s\in S_0^0(1)(\R^{2d}\times \T^d)$ 
we have for some $C>0$
\begin{equation}\label{10prop4}
\bigl| S_1\bigr|\leq  C \sqrt{\ep} \bigl\|\id_E e^{-\frac{C_0}{2\ep}(\,. \,+R)^2}v_j \bigr\|_{\ell^2}\, 
\bigl\|\id_E e^{-\frac{C_0}{2\ep}(\,. \,+R)^2}v_k\bigr\|_{\ell^2} \leq \sqrt{\ep} C e^{-\frac{C_0}{\ep}R_E^2}
\end{equation}
for $R_E:= \min_{x\in E} |x_d - R|$. Thus taking $R$ large enough such that $R_E>S_{jk}$ and inserting \eqref{10prop4} and 
\eqref{7prop4} in \eqref{3prop4} proves the proposition.

\end{proof}

In the next proposition we show that, modulo a small error, the interaction term only depends on a small neighborhood
of the point or manifold respectively where the geodesics between $x^j$ and $x^k$ intersect $\Hd_d$. Since the proof is analogue, we
discuss the point and manifold case simultaneously.

\begin{prop}\label{prop5}
Let $\Psi\in \Ce_0^\infty (\stackrel{\circ}{M}_{j}\cap \stackrel{\circ}{M}_{k}\cap E)$ denote a cut-off-function near
$y_0\in \Hd_d$ (or $G_0\subset \Hd_d$ respectively) such that $\Psi=1$ in a neighborhood $U_\Psi$ of $y_0$ (or $G_0$ respectively)
and for some $C>0$ 
\begin{equation}\label{1prop5}
\frac{C_0}{2}x_d^2 + d^j(x) + d^k(x) - S_{jk} >C \, , \qquad x\in \supp (1-\Psi)\, .
\end{equation}
Then, for the restriction $\Psi^\ep:= r_\ep \Psi$ of $\Psi$ to the lattice $\disk$ (see \eqref{restrict}),
\begin{equation}\label{0prop5}
w_{jk} =  \ep \skpd{Q_0 \Psi^\ep v_j}{\Psi^\ep v_k} + O\Bigl(\ep^\infty e^{-\frac{1}{\ep}S_{jk}}\Bigr)\; .
\end{equation}
\end{prop}

\begin{proof}
Using Proposition \ref{prop4} and the notation \eqref{8prop4}, \eqref{9prop4} together with \eqref{11aprop4} we have
\begin{align}\label{2prop5}
w_{jk}&=  \skpd{\tilde{Q}_0 \id_E \tilde{v}_{j,0}}{ \id_E \tilde{v}_{k,0}} + O\Bigl(\ep^\infty e^{-\frac{1}{\ep}S_{jk}}\Bigr)\\
&= \ep \skpd{Q_0 \Psi^\ep  \id_E v_j}{\Psi^\ep  \id_E v_k} + R_1 + R_2 + R_3 +  O\Bigl(\ep^\infty e^{-\frac{1}{\ep}S_{jk}}\Bigr)\nonumber
\end{align}
where, using $\id_E \Psi = \Psi$,
\begin{align}\label{3prop5}
R_1 &= \skpd{\tilde{Q}_0 (1-\Psi^\ep ) \id_E \tilde{v}_{j,0}}{ \Psi^\ep  \tilde{v}_{k,0}}\\
R_2 &= \skpd{\tilde{Q}_0\Psi^\ep  \tilde{v}_{j,0}}{(1-\Psi^\ep ) \id_E \tilde{v}_{k,0}}\\
R_3 &= \skpd{\tilde{Q}_0 (1-\Psi^\ep )\id_E \tilde{v}_{j,0}}{ (1-\Psi^\ep )\id_E \tilde{v}_{k,0}}\; .
\end{align}
To estimate $|R_1|$ we write
\begin{align*}
\bigl|R_1\bigr| &= \Bigl| \skpd{e^{-\frac{1}{\ep}(d^k + d^j)}(1-\Psi^\ep )e^{\frac{d^j}{\ep}} \id_E \tilde{v}_{j,0}}
{\chi_E e^{\frac{d^k}{\ep}}\tilde{Q}_0^*e^{-\frac{d^k}{\ep}}e^{\frac{d^k}{\ep}}\Psi^\ep  \tilde{v}_{k,0}}\Bigr|\\
&\leq \Bigl\|e^{-\frac{1}{\ep}(d^k + d^j + \frac{C_0}{2}(.)_d^2)} (1-\Psi^\ep ) e^{\frac{d^j}{\ep}} \id_E v_j\Bigr\|_{\ell^2} 
\Bigl\| \chi_E e^{\frac{d^k}{\ep}}\tilde{Q}_0^*e^{-\frac{d^k}{\ep}}e^{\frac{d^k}{\ep}} \Psi^\ep \tilde{v}_{k,0}\Bigr\|_{\ell^2}
\end{align*}
where $\chi_E$ denotes a cut-off function as introduced above Proposition \ref{prop3}.
Since by \eqref{9prop4} 
\[ \tilde{Q}_0^* = \widetilde{\Op}_\ep^T \Bigl(\sqrt{\frac{C_0 \ep}{\pi}} q_0^*\Bigr)\quad\text{for}\quad q_0^*(x,y,\xi;\ep)= 
q_0(y,x,\xi;\ep)\in S_0^0(1)(\R^{2d}\times \T^d)\, ,\] 
it follows from Proposition \ref{prop3app} that
$\chi_E e^{\frac{d^k}{\ep}}\tilde{Q}_0^*e^{-\frac{d^k}{\ep}}$ is the 0-quantization of a symbol 
$q_{0,d^k,0}\in S^{\frac{1}{2}}_0(1)(\R^d\times \T^d)$.
Thus by Corollary \ref{cor1app} and \eqref{1prop5}, for some $C, C'>0$,
\begin{equation}\label{4prop5}
 |R_1| \leq e^{-\frac{S_{jk}+C}{\ep}} C' \sqrt{\ep} \Bigl\|e^{\frac{d^j}{\ep}} v_j\Bigr\|_{\ell^2}  \Bigl\|e^{\frac{d^k}{\ep}} v_k\Bigr\|_{\ell^2} = 
 O\Bigl(\ep^\infty e^{-\frac{S_{jk}}{\ep}}\Bigr)
\end{equation}
where the last estimate follows from \cite{kleinro5}, Proposition 3.1.\\ 
Similar arguments show $|R_2| = O(\ep^\infty e^{-\frac{S_{jk}}{\ep}}) = |R_3|$, 
thus by \eqref{2prop5} this finishes the proof. 

\end{proof}

In the next step, we show that modulo the same error term, the Dirichlet eigenfunctions $v_m,\, m=j,k,$ 
can be replaced by the approximate eigenfunctions 
$\hat{v}_m^\ep$  given  in \eqref{hatvm}.
We showed in \cite{kleinro5}, Theorem 1.7, that for some smooth functions $b^m, b^m_\ell$, compactly supported in a 
neighborhood of $M_m$, 
the approximate eigenfunctions $\hat{v}^\ep_m\in \ell^2(\disk)$ are given by
the restrictions to $\disk$ of 
\begin{equation}\label{approx}
\hat{v}_m := \ep^{\frac{d}{4}} e^{-\frac{d^m}{\ep}} b^m \, , \qquad\text{where}\quad\
 b^m \sim \sum_{\ell\geq M} \ep^\ell b^m_\ell
\end{equation}
(using the notation in \cite{kleinro5}, these restrictions are $\hat{v}^\ep_{m,1,0}$).
In \cite{kleinro5}, Theorem 1.8 we proved that for any $K$ compactly supported in $M_m$ the estimate
\begin{equation}\label{approxl2}
 \Bigl\| e^{\frac{d^m}{\ep}}(v_m - \hat{v}^\ep_m)\Bigr\|_{\ell^2(K)} = O\bigl(\ep^\infty)\; .
\end{equation}
holds. Using \eqref{approxl2} we get the following Proposition.

\begin{prop}\label{prop6}
Let $ \hat{v}^\ep_m\in \ell^2(\disk),\, m=j,k,$ denote the approximate eigenfunctions of $H_\ep$ in $M_m$ constructed in 
\cite{kleinro5}, Theorem 1.7, then, for $\Psi^\ep$ as defined in Proposition \ref{prop5},
\begin{equation}\label{0prop6}
w_{jk} =  \ep \skpd{Q_0 \Psi^\ep  \hat{v}^\ep_j}{\Psi^\ep  \hat{v}^\ep_k } + O\Bigl(\ep^\infty e^{-\frac{1}{\ep}S_{jk}}\Bigr)\; .
\end{equation}
\end{prop}

\begin{proof}
 By Proposition \ref{prop5}
\begin{multline}\label{1prop6}
w_{jk} =  \ep \skpd{Q_0 \Psi^\ep \hat{v}^\ep_j}{\Psi^\ep \hat{v}^\ep_k } + \ep \skpd{Q_0 \Psi^\ep (v_j - \hat{v}^\ep_j)}{\Psi^\ep v_k }\\ +
\ep \skpd{Q_0 \Psi^\ep \hat{v}^\ep_j}{\Psi^\ep (v_k - \hat{v}^\ep_k) }+ O\Bigl(\ep^\infty e^{-\frac{1}{\ep}S_{jk}}\Bigr)\; .
\end{multline}
Using the notation \eqref{8prop4}, \eqref{9prop4} with $\tilde{u}:= \tilde{u}_0$ together with \eqref{11aprop4}, we can write
\begin{align}
 \bigl| \ep \skpd{Q_0 \Psi^\ep (v_j - \hat{v}^\ep_j)}{\Psi^\ep v_k }\bigr| &= 
\bigl|  \skpd{\tilde{Q}_0 \Psi^\ep (\tilde{v}_j - \tilde{\hat{v}}^\ep_j)}{\Psi^\ep \tilde{v}_k }\bigr| \nonumber\\
 &= \bigl| \skpd{\chi_E e^{-\frac{d^k}{\ep}}\tilde{Q}_0 e^{\frac{d^k}{\ep}} \chi_E \Psi^\ep e^{-\frac{d^k + d^j}{\ep}}
 e^{\frac{d^j}{\ep}}(\tilde{v}_j - \tilde{\hat{v}}^\ep_j)}{e^{\frac{d^k}{\ep}}\Psi^\ep \tilde{v}_k }\bigr|\nonumber \\
 &\leq e^{-\frac{S_{jk}}{\ep}}\sqrt{\ep} C \bigl\| \Psi^\ep e^{- \frac{C_0 (.)_d^2}{\ep}} 
e^{\frac{d^j}{\ep}}(v_j - \hat{v}^\ep_j)\bigr\|_{\ell^2} 
 \bigl\| \Psi^\ep e^{\frac{d^k}{\ep}}v_k \bigr\|_{\ell^2}\; ,\label{2prop6}
\end{align}
where, analog to \eqref{4prop5}, the last estimate follows from Proposition \ref{prop3} together with Corollary \ref{cor1app} for the operator
$\chi_E e^{-\frac{d^k}{\ep}}\tilde{Q}_0 e^{\frac{d^k}{\ep}} \chi_E$. Since $\Psi$ is compactly supported in 
$\stackrel{\circ}{M}_j$, we get by \eqref{approxl2} 
for any $N\in\N$
\begin{equation}\label{3prop6}
 \bigl\| \Psi^\ep e^{- \frac{C_0 (.)_d^2}{\ep}} e^{\frac{d^j}{\ep}}(v_j -\hat{v}^\ep_j)\bigr\|_{\ell^2} \leq 
 \bigl\| e^{\frac{d^j}{\ep}}(v_j - \hat{v}^\ep_j)\bigr\|_{\ell^2(\supp \Psi)} = O(\ep^N)\; . 
\end{equation}
Since by \cite{kleinro5}, Proposition 3.1
\begin{equation}\label{4prop6}
 \bigl\| \Psi^\ep e^{\frac{d^k}{\ep}}v_k \bigr\|_{\ell^2} \leq C \ep^{-N_0}
\end{equation}
for some $C>0$, $N_0\in\N$, we can conclude by inserting \eqref{4prop6} and \eqref{3prop6} in \eqref{2prop6}
\begin{equation}\label{5prop6}
 \bigl| \ep \skpd{Q_0 \Psi^\ep (v_j - \hat{v}^\ep_j)}{\Psi^\ep v_k }\bigr| = O\Bigl(\ep^\infty e^{-\frac{S_{jk}}{\ep}}\Bigr)\; .
\end{equation}
Analog arguments show
\begin{equation}\label{6prop6}
 \bigl| \ep \skpd{Q_0 \Psi^\ep \hat{v}^\ep_j}{\Psi^\ep (v_k - \hat{v}^\ep_k)}\bigr| = O\Bigl(\ep^\infty e^{-\frac{S_{jk}}{\ep}}\Bigr)\; .
\end{equation}
Inserting \eqref{5prop6} and \eqref{6prop6} in \eqref{1prop6} gives \eqref{0prop6}.

\end{proof}

Proposition \ref{prop6} together with \eqref{approx}, \eqref{8prop4} and \eqref{11aprop4} lead at once to the following corollary.

\begin{cor}\label{cor1}
For $b^j, b^k\in \Ce_0^\infty (\R^d\times (0,\ep_0])$ as given in \eqref{hatvm}, $\Psi$ as defined in Proposition \ref{prop5}
and the restriction map $r_\ep$ given in \eqref{restrict} we have
\begin{equation}\label{0cor1}
 w_{jk} =  \ep^{\frac{d}{2}} e^{-\frac{S_{jk}}{\ep}} \skpd{\hat{Q}_0 r_\ep \Psi b^j}{ e^{-\frac{\varphi}{\ep}}\Psi b^k} 
+ O\Bigl(\ep^\infty e^{-\frac{1}{\ep}S_{jk}}\Bigr)
\end{equation}
where for $\tilde{Q}_0$ defined in \eqref{9prop4} we set
\begin{align}\label{1cor1}
 \varphi (x) &:= d^j(x) + d^k(x) + C_0|x_d|^2 - S_{jk} \\
 \hat{Q}_0 &:= e^{\frac{1}{2\ep}C_0(.)_d^2} e^{\frac{d^j}{\ep}} \tilde{Q}_0 
e^{-\frac{d^j}{\ep}}e^{-\frac{1}{2\ep}C_0(.)_d^2} \; .\label{2cor1}
\end{align}
\end{cor}

\begin{rem}\label{Remcor1}
\ben
\item Setting $\psi (x) = \frac{1}{2\ep}C_0 x_d^2 + \frac{1}{\ep} d^j (x)$, it follows from Proposition \ref{prop3app} together with 
\eqref{2cor1} and \eqref{9prop4} that 
the operator $\hat{Q}_0$ is the $0$-quantization  of a symbol 
$\hat{q}_{\psi}\in S_0^{\frac{1}{2}}(1)(\R^d\times \T^d)$, which has an asymptotic expansion, in particular
\begin{equation}\label{0remcor1}
\hat{Q}_0 = \Op_{\ep}^\T\bigl(\hat{q}_{\psi}\bigr)\, , \qquad \hat{q}_\psi (x, \xi; \ep) \sim \ep^{\frac{1}{2}} \sum_{n=0}^\infty \ep^n \hat{q}_{n,\psi}(x, \xi)\; . 
\end{equation}
Modulo $S^{\frac{3}{2}}_0(1)(\R^{d}\times \T^d)$, the symbol $\hat{q}_\psi$ is given by
\begin{equation}\label{0aremcor1}
\ep^{\frac{1}{2}} \hat{q}_{0,\psi}(x, \xi) = \sqrt{\frac{\ep C_0}{\pi}} q_{0,0} \bigl(x, x, \xi - i\nabla d^j(x) - iC_0 x_d e_d\bigr) 
\end{equation}
where $e_d$ denotes the 
unit vector in $d$-direction (see Proposition \ref{prop3}). At the intersection point or intersection manifold, 
i.e. for $y=y_0$ or $y\in G_0$ respectively,
by \eqref{8prop3} the leading order of the symbol is given by 
\begin{equation}\label{1.remcor1}
\ep^{\frac{1}{2}}\hat{q}_{0,\psi}(y,\xi)= \frac{1}{i}\sqrt{\frac{\ep C_0}{\pi}} \partial_{\xi_d} t_0 \bigl(y, \xi - i\nabla d^j (y)\bigr) = 
-\sqrt{\frac{\ep C_0}{\pi}}\sum_{\eta\in\Z^d} \tilde{a}_\eta(y)  \eta_d e^{-i\eta\cdot (\xi-i\nabla d^j (y))}
\end{equation}
where $\tilde{a}_{\eta} = a_{\ep\eta}^{(0)}$ for $\eta\in\Z^d$.
\item By Corollary \ref{cor1} we can write
\begin{equation}\label{1thm1}
 w_{jk} = \ep^{\frac{d}{2}} e^{-\frac{S_{jk}}{\ep}} \sum_{x\in\disk} e^{-\frac{\varphi(x)}{\ep}} 
\bigl( \hat{Q}_0 r_\ep \Psi b^j\bigr)(x) \bigl(\Psi b^k\bigr)(x) + O\Bigl(\ep^\infty e^{-\frac{S_{jk}}{\ep}}\Bigr)\; .
\end{equation}
\item In the setting of Hypothesis \ref{hypgeo2}, we have $\varphi|_{G_0} = 0$ and moreover, since $d^j + d^k$ is minimal on $G_0$, 
$\nabla \varphi|_{G_0} = 0$ and $\varphi (x)>0$ for 
$x\in \supp \Psi \setminus G_0$.
\een
\end{rem}

\section{Proof of Theorem \ref{wjk-expansion}}\label{section3}

A key element of the proofs of both theorems is replacing the sum on the right hand side of \eqref{1thm1} by an integral, up to a small error. 
Here we follow arguments
from \cite{giacomo}. 

In particular, in the case of just one minimal geodesic, we can use  
Corollary C.2 in \cite{giacomo}, telling us the following: Let $a\in\Ce_0^\infty (\R^n, \R)$ and 
$\psi\in \Ce^\infty (\R^n, \R)$ be such
that $\psi(x_0)=0$, $D^2\psi (x_0)>0$ and $\psi(x)>0$ for $x\in \supp a\setminus \{x_0\}$ for some $x_0\in\R^n$. 
Then there exists a sequence $(J_k)_{k\in\N}$ in $\R$ such that
\begin{equation}\label{2thm1}
\ep^{\frac{d}{2}} \sum_{x\in\disk} a(x)e^{-\frac{\psi(x)}{\ep}}  \sim \sum_{k=0}^\infty \ep^k J_k\quad
 \text{where}\quad J_0 = \frac{(2\pi)^{\frac{d}{2}} a(x_0)}{\sqrt{\det D^2\psi (x_0)}}\; .
\end{equation}
We observe that the proof of \eqref{2thm1} for $a(x)$ being independent of $\ep$ immediately generalizes to an asymptotic expansion
$a(x,\ep) \sim \sum \ep^k a_k(x)$.

In order to apply \eqref{2thm1} to the right hand side of \eqref{1thm1} we have to verify the assumptions above for 
$\psi = \varphi $ defined in \eqref{1cor1} and for some
$a\in\Ce_0^\infty$ which is equal to $\Psi b^k \bigl(\hat{Q}_0 r_\ep \Psi b^j\bigr)$ on $\disk$ and has an asymptotic expansion
in $\ep$.\\

It follows directly from its definition that $\varphi(y_0)=0$. Since $d^j (x)+ d^k (x)- S_{jk}> 0$ in $E\setminus \gamma_{jk}$ by triangle 
inequality and  
$x_d^2>0$ for all $x\in \gamma_{jk}, x\neq y_0$, it follows that $\varphi (x)>0$ for $x\in \supp \Psi \setminus \{y_0\}$.

To see the positivity of $D^2\varphi (y_0)$ we first remark that by Hypothesis \ref{hypgeo1} $d^j + d^k$, restricted to 
$\Hd_d$, has a positive Hessian at $y_0$, which we denote by
$D^2_\perp (d^j + d^k)(y_0)$. Since furthermore $d^j + d^k$ is constant along the geodesic, it
follows that the full Hessian $D^2(d^j + d^k)(y_0)$ has $d-1$ positive eigenvalues and the eigenvalue zero. The Hessian of $C_0 x_d^2$ 
at $y_0$ is diagonal and the only
non-zero element is $\partial_d^2 (C_0 x_d^2) = 2C_0>0$. Thus the Hessian $D^2 \varphi (y_0)$ is a non-negative quadratic form. 
In order to show that it is in fact positive, we
analyze its determinant. Writing the last column as  the sum $\nabla \partial_d (d^j + d^k)(y_0) + v$ where $v_k=0$ for 
$1\leq k \leq d-1$ and $v_d=2 C_0$ we 
get
\begin{align} 
\det D^2\varphi (y_0) &= \det D^2 (d^j + d^k)(y_0) + 
\det \begin{pmatrix} D^{2}_\perp(d^j + d^k)(y_0) & 0 \\ * & 2 C_0 \end{pmatrix}\nonumber\\
 &= 2 C_0 \det D^{2}_\perp(d^j + d^k)(y_0) >0 \label{6thm1}
\end{align}
where the second equality follows from the fact that one eigenvalue of $D^2(d^j + d^k)(y_0)$ is zero as discussed above and thus 
its determinant is zero.
This proves that $D^2\varphi (y_0)$ is non-degenerate and thus we get $D^2\varphi (y_0) >0$. 

By Proposition \ref{prop1app}, Remark \ref{remprop1app} and \eqref{0remcor1} the operator 
$\hat{Q}_0= \Op_{\ep}^\T(\hat{q}_\psi)$ on $\ell^2(\disk)$ (multiplied from the right by the restriction operator $r_\ep$) is equal
to the restriction of the operator $\Op_{\ep}(\hat{q}_\psi)$ on $L^2(\R^d)$. Here we consider $\hat{q}_\psi$ as periodic element of the symbol class
$S_0^{\frac{1}{2}}(1)\bigl(\R^d\times \R^d\bigr)$. In particular, for $x\in\disk$ we have 
\begin{equation}\label{13thm1}
 \Psi b^k(x) \hat{Q}_0 r_\ep \Psi b^j (x) = \Psi b^k(x) \Op_{\ep} (\hat{q}_\psi) \Psi b^j (x)
\end{equation}
where $r_\ep$ denotes the restriction to the lattice $\disk$ defined in \eqref{restrict}.
We therefore set 
\begin{equation}\label{14thm1}
a(x;\ep) := \Psi b^k(x) \Op_{\ep} \bigl(\hat{q}_\psi\bigr) \Psi b^j (x) \, , \qquad x\in\R^d\; .
\end{equation}
Then $a = \Psi b^k \bigl(\hat{Q}_0 r_\ep \Psi b^j\bigr)$ on $\disk$ and $a(.; \ep)\in\Ce_0^\infty(\R^d)$, because 
$\Psi, b^k, b^j\in \Ce_0^\infty(\R^d)$ (see e.g. \cite{dima}, which gives that $\Op_{\ep} \bigl(\hat{q}_\psi\bigr)$ maps $\mathcal{S}$ to $\mathcal{S}$).

Next we show that $a(x; \ep)$ has an asymptotic expansion in $\ep$. It suffices to show this for 
$\Op_{\ep} (\hat{q}_\psi) \Psi b^j$.

It follows from the asymptotic expansions of $\hat{q}_\psi$ and $b^j$ in \eqref{0remcor1} and \eqref{hatvm} that
\begin{align}
\Op_{\ep} (\hat{q}_\psi) \Psi b^j (x;\ep) &\sim 
\sum_{n=0}^\infty \sum_{\natop{\ell\in \Z/2}{\ell\geq -N_j}} \ep^{\frac{1}{2}+n+\ell} \Op_{\ep} (\hat{q}_{n,\psi}) \Psi b^j_\ell (x)\nonumber \\
& \sim \sum_{n=0}^\infty \sum_{\natop{\ell\in \Z/2}{\ell\geq -N_j}} \ep^{\frac{1}{2}+n+\ell} 
(2\pi\ep )^{-d} \int_{\R^{2d}} e^{\frac{i}{\ep}(y-x)\xi}
\hat{q}_{n,\psi}(x ,\xi) \Psi b^j_\ell (y) \, dy \, d\xi \nonumber\\
&\sim \sum_{n=0}^\infty \sum_{\natop{\ell\in \Z/2}{\ell\geq -N_j}}\sum_{m=0}^\infty \ep^{\frac{1}{2}+n+\ell+m} 
(2\pi)^{-d} \int_{\R^{2d}} e^{i(y-x)\zeta}
\hat{q}_{m,n,\psi}(x) \zeta^m \Psi b^j_\ell (y) \, dy \, d\zeta \label{5thm1}
\end{align}
where the last equality follows from the analyticity of $\hat{q}_\psi$ with respect to $\xi$, using the substitution $\zeta \ep = \xi$.
The functions $\hat{q}_{m,n,\psi}(x)$ are the coefficients of the expansion of $\hat{q}_{n,\psi}(x,\cdot)$ into a
convergent power series in $\xi$ at zero.

Thus we can apply 
\eqref{2thm1} to \eqref{1thm1}, which gives
\begin{equation}\label{3thm1}
 w_{jk} \sim e^{-\frac{S_{jk}}{\ep}} \sum_{k=0}^\infty \ep^k J_k 
 \end{equation}
where $J_0$ is the leading order term of 
\begin{equation}\label{4thm1}
\tilde{J}_0 = \frac{(2\pi)^{\frac{d}{2}}}{\sqrt{\det D^2\varphi (y_0)}}b^k(y_0) (\Op_{\ep}(\hat{q}_\psi) \Psi b^j)(y_0;\ep)  \; .
\end{equation}

By \eqref{1.remcor1} it follows that
\begin{equation}\label{20thm1}
\hat{q}_{0,0,\psi}(y_0) = -\sqrt{\frac{C_0}{\pi}} \sum_{\eta\in\Z^d} \tilde{a}_\eta(y_0)  \eta_d e^{-\eta\cdot \nabla d^j (y_0)} \; .
\end{equation}
Thus, by \eqref{5thm1} and Fourier inversion formula, the leading order term of $ (\Op_{\ep}(\hat{q}_\psi) \Psi b^j)(y_0;\ep)$ is given by 
\begin{multline}\label{21thm1}
  \ep^{\frac{1}{2}- N_j}\hat{q}_{0,0,\psi}(y_0) (2\pi)^{-d} \int_{\R^{2d}} e^{i(y-y_0)\zeta} \Psi b^j_{-N_j} (y) \, dy \, d\zeta \\
  = -  \ep^{\frac{1}{2}- N_j}\sqrt{\frac{C_0}{\pi}} 
\sum_{\eta\in\Z^d} \tilde{a}_\eta(y_0)  \eta_d e^{-\eta\cdot \nabla d^j (y_0)} \Psi b^j_{-N_j} (y_0)
\end{multline}

From \eqref{21thm1},\eqref{6thm1}, \eqref{4thm1} and \eqref{3thm1} it follows that $w_{jk}$ has the stated asymptotic expansion 
(where $J_0 = I_0 \ep^{\frac{1}{2}-(N_j + N_k)}$) with 
leading order 
\begin{equation}\label{23thm1}
I_0 = - \frac{(2\pi)^{\frac{d-1}{2}}}{\sqrt{\det D^{2}_\perp(d^j + d^k) (y_0)}} b^k_{-N_k}(y_0) 
 \sum_{\eta\in\Z^d} \tilde{a}_\eta(y_0) \eta_d e^{-\eta\cdot \nabla d^j (y_0)} b^j_{-N_j}(y_0) \; .
\end{equation}
Writing 
\begin{multline}\label{22thm1}
 \sum_{\eta\in\Z^d} \tilde{a}_\eta(y_0) \eta_d e^{-\eta\cdot \nabla d^j (y_0))} 
= \frac{1}{2}\sum_{\eta\in\Z^d} \bigl( \tilde{a}_{\eta}(y_0) \eta_d e^{-\eta\cdot \nabla d^j (y_0)} + \tilde{a}_{-\eta}(y_0)(-\eta_d) 
e^{\eta\cdot \nabla d^j (y_0)}\bigr) \\
= \sum_{\eta\in\Z^d} \tilde{a}_{\eta}(y_0) \eta_d \sinh \bigl(\eta\cdot \nabla d^j (y_0)\bigr)
\end{multline}
where in the last step we used $\tilde{a}_\eta(y_0) = \tilde{a}_{-\eta}(y_0)$ (see \eqref{agammasym}) and inserting \eqref{22thm1} 
into \eqref{23thm1} gives
\eqref{0thm1}.
Note that all $I_k$ are indeed real (since $w_{jk}$ is real).

\qed

\section{Proof of Theorem \ref{wjk-expansion2}}\label{section4}

{\sl Step 1:} As in the previous proof, we start proving that the sum in the formula \eqref{1thm1} for the interaction term $w_{jk}$ can, up to small error, be 
replaced by an integral. 
This can be done using the following lemma, which is proven e.g. in \cite{giacomo}, Proposition C1,
using Poisson's summation formula.

\begin{Lem}\label{LemmaC1}
For $h>0$ let $f_h$ be a smooth, compactly supported function on $\R^d$ with the property: there exists $N_0\in \N$ such that
for all $\alpha\in\N^d, |\alpha|\geq N_0$ there exists a $h$-independent constant $C_\alpha$ such that
\begin{equation}\label{1LemC1}
 \int_{\R^d} |\partial^\alpha f_h (y)|\, dy \leq C_\alpha\, .
\end{equation}
Then
\begin{equation}\label{2LemC1}
h^d \sum_{y\in h\Z^d} f_h(y) = \int_{\R^d} f_h(y) \, dy + O( h^\infty)\, , \qquad (h\to 0) \; .
\end{equation}
\end{Lem}

We shall verify that Lemma \ref{LemmaC1} can be used to evaluate the interaction matrix as given in \eqref{1thm1}. For $a$ given by
\eqref{14thm1}
we claim that for any $\alpha_1\in \N^d$ 
there is a constant $C_{\alpha_1}$ such that
\begin{equation}\label{1.1thm2}
\sup_{x\in\R^d} \bigl| \partial^{\alpha_1}_x a(x; \ep) \bigr| \leq  C_{\alpha_1}\ep^{\frac{1}{2}} \; .
\end{equation}
Clearly it suffices to prove 
\begin{equation}\label{14.1thm2}
\sup_{x\in\R^d} \bigl| \Psi(x) \partial^{\alpha_1}_x \Op_\ep(\hat{q}_\psi) \Psi b^j(x; \ep) \bigr| \leq  C_{\alpha_1}\ep^{\frac{1}{2}} 
\end{equation}
or, by Sobolev`s Lemma (see i.e. \cite{Folland}), for all $\beta\in\N^d$ with $|\beta|\leq \frac{d}{2}+ 1$
\begin{equation}\label{15.1thm2}
\bigl\| \Psi  \partial^{\beta + \alpha_1} \Op_\ep(\hat{q}_\psi) \Psi b^j(\,.\,; \ep) \bigr\|_{L^2} \leq  C\ep^{\frac{1}{2}} \; .
\end{equation}
Setting for $0\leq \ell \leq |\beta +\alpha_1|$  
\[ c_\ell(\xi) := \sum_{\natop{\gamma\in\N^d}{|\gamma|=\ell}} \frac{1}{\gamma!}\partial_\xi^\gamma \xi^{\beta + \alpha_1}  \quad\text{and}\quad 
\hat{q}_{\psi,\ell}(x, \xi; \ep) := \sum_{\natop{\gamma\in\N^d}{|\gamma|=\ell}} \partial_x^\gamma \hat{q}_\psi(x, \xi; \ep)\; ,
\]
we have by symbolic calculus (see e.g. \cite{martinez}, Thm.2.7.4 )
\begin{align} 
\partial^{\beta + \alpha_1} \Op_\ep(\hat{q}_\psi) &= \Bigl(\frac{i}{\ep}\Bigr)^{|\beta + \alpha_1|} \Op_\ep(c_0)  \Op_\ep(\hat{q}_\psi)\nonumber\\
& = \Bigl(\frac{i}{\ep}\Bigr)^{|\beta + \alpha_1|} \sum_{\ell=0}^{|\beta + \alpha_1|} 
\Op_\ep(\hat{q}_{\psi, \ell}) \Op_\ep(c_\ell)\Bigl(\frac{\ep}{i}\Bigr)^{\ell}\nonumber\\
&= \sum_{\ell=0}^{|\beta + \alpha_1|} \Op_\ep(\hat{q}_{\psi, \ell}) c_\ell (\partial_\xi)\label{16.1thm2}
\end{align}
where in the last step we used that $c_\ell(\xi)$ is homogeneous of degree $|\beta + \alpha_1| - \ell$.
Since $\Psi b^j$ is smooth and $\hat{q}_\psi\in S_0^\frac{1}{2}(1)\bigl(\R^{2d}\bigr)$, \eqref{15.1thm2} (and thus \eqref{1.1thm2}) 
follows from \eqref{16.1thm2} 
together with the Theorem of Calderon and Vaillancourt (see e.g. \cite{dima}).\\

Then for $\varphi$ and $a$ given by \eqref{1cor1} and \eqref{14thm1} respectively and for $h=\sqrt{\ep}$, we set $y=\frac{x}{h}$ and  
\begin{equation}\label{2.1thm2}
 f_h(y) :=  h^\ell e^{-\varphi_h (y)} A_h(y) \quad\text{where}\quad \varphi_h(y) := \frac{\varphi (hy)}{h^2} \quad \text{and}\quad 
A_h(y) := a(hy; h^2)\; .
\end{equation}
Then for $\alpha\in\N^d$
\begin{equation}\label{5.1thm2}
 \partial^\alpha f_h  =: h^\ell g_{h,\alpha} e^{-\varphi_h}
\end{equation}
where $g_{h,\alpha}$ is a sum of products, where the factors are given by
$\partial^{\alpha_1} A_h$ and $\partial^{\alpha_2}\varphi_h, \ldots , \partial^{\alpha_m}\varphi_h$ for 
partitions $\alpha_1, \ldots \alpha_m\in\N^d$ of $\alpha$, i.e. $\sum_r \alpha_r = \alpha$. 
By \eqref{1.1thm2} and \eqref{2.1thm2} we have for some $C_{\alpha_1}$ independent of $h$
\begin{equation}\label{6.1thm2}
\sup_{y\in\R^d} \bigl| \partial^{\alpha_1} A_h(y) \bigr| \leq  h^{1 + |\alpha_1|} C_{\alpha_1}\; .
\end{equation}
In order to analyze $|\partial^{\alpha_2} \varphi_h|$, we remark that
Taylor expansion at $y_0$ yields for $\beta\in \N^d$
\begin{multline}\label{7.1thm2}
\partial^{\beta} \varphi_h (y) = h^{|\beta| - 2} (\partial^{\beta}\varphi)(hy) 
=  h^{|\beta| - 2} (\partial^{\beta}\varphi)(hy_0) + 
 h^{|\beta| - 1} (\nabla \partial^{\beta}\varphi)|_{hy_0}(y- y_0)  \\
+  h^{|\beta|} \int_0^1 \frac{(1-t)^2}{2} (D^2 \partial^{\beta}\varphi)|_{h(y_0 + t(y-y_0))}[y-y_0]^2 \, dt\; .
\end{multline} 
Since for $y\in \supp A_h, y_0\in h^{-1}G_0$ the curve $t\mapsto h(y_0 + t(y-y_0))$ lies in a compact set, it follows from \eqref{7.1thm2} together with 
Remark \ref{Remcor1},(3),
that for some $C_{\beta}$ and for $N_\beta = \max \{0, |\beta|-2\}$
\begin{equation}\label{8.1thm2}
|\partial^{\beta} \varphi_h (y)| \leq C_{\beta} h^{N_{\beta}} \bigl( 1+ |y-y_0|^2\bigr) \, , \qquad y_0\in h^{-1}G_0\, , \; y\in \supp A_h\; .
\end{equation}
Thus using the above mentioned structure of $g_{h,\alpha}$ we get 
\begin{equation}\label{9.1thm2}
\bigl| g_{h, \alpha}(y)\bigr| \leq C_\alpha h \Bigl( 1+ \bigl|y - y_0\bigr|^{2 |\alpha|}\Bigr)
\end{equation}
where $C_\alpha$ is uniform for $y\in \supp A_h$ and $y_0\in h^{-1}G_0$. Taking the infimum over all $y_0$ on 
the right hand side of \eqref{9.1thm2} we get
\begin{equation}\label{9a.1thm2}
\bigl| g_{h, \alpha}(y)\bigr| \leq C_\alpha h \Bigl( 1+ \bigl( \dist (y,  h^{-1}G_0) \bigr)^{2 |\alpha|}\Bigr)
\end{equation}
Since by Hypothesis \ref{hypgeo2a} $G$ is non-degenerate at $G_0$ we have for some $C>0$
\[ \varphi (x) \geq C \dist (x, G_0)^2 \]
and therefore
\begin{equation}\label{10.1thm2}
\varphi_h (y) \geq C \frac{1}{h^2} \dist (hy, G_0)^2 = C \dist (y, h^{-1} G_0)^2\; .
\end{equation}
Combining \eqref{5.1thm2}, \eqref{9a.1thm2} and \eqref{10.1thm2} gives
\begin{align}
\int_{\R^d} \bigl| \partial^\alpha f_h (y) \bigr| \, dy &= 
  h^\ell \int_{\R^d} \bigl| g_{h, \alpha} e^{- \varphi_h (y)}\bigr|\, dy\nonumber\\
&\leq C_\alpha h^{\ell + 1}  \int_{\supp A_h} e^{- C \dist(y,h^{-1} G_0)^2}  \Bigl( 1 + \bigl(\dist(y, h^{-1} G_0)\bigr)^{2|\alpha|}\Bigr) \, dy \nonumber\\
&=C_\alpha h^{\ell + 1 - d} \int_{\supp \Psi} e^{-\frac{C}{h^2} \dist(x, G_0)^2}  \Bigl( 1 + h^{-2|\alpha|} \bigl(\dist(x, G_0)\bigr)^{2|\alpha|}\Bigr) 
\, dx  \label{11.1thm2}
\end{align}
where in the last step we used the substitution $x = h y$.

Using the Tubular Neighborhood Theorem, there is a diffeomorphism 
\begin{equation}\label{tubk} 
k: \supp \Psi \rightarrow G_0\times (-\delta, \delta)^{d-\ell} \, , \quad k(x) = (s,t)\; .
\end{equation}
Here $\delta>0$ must be chosen adapted to $\supp \Psi$, which is an arbitrary small neighborhood of $G_0$. 
Denoting by $d\sigma$ the Euclidean surface element on $G_0$, the right hand side of \eqref{11.1thm2} can thus be estimated from above by
\begin{align}
C'_\alpha h^{\ell + 1 - d}\int_{G_0\times (-\delta, \delta)^{d-\ell}} e^{-\frac{C}{h^2}t^2}  \Bigl( 1 + \Bigl(\frac{t}{h}\Bigr)^{2|\alpha|}\Bigr) 
\, d\sigma(s) \, dt \nonumber\\
\leq \tilde{C}_\alpha h \int_{\R^{d-\ell}} e^{-C \tau^2}\bigl( 1 + |\tau|^{2|\alpha|}\bigr) \, d\tau \leq \hat{C}_\alpha \label{11a.1thm2}
\end{align}
where in the last step we used that $G_0$ was assumed to be compact and the substitution $t = \tau h$.

By \eqref{11.1thm2} and \eqref{11a.1thm2} we can use Lemma \ref{LemmaC1} for $f_h$ given in \eqref{2.1thm2} and thus we have by \eqref{1thm1} together
with \eqref{13thm1} and \eqref{14thm1} 
\begin{equation}\label{12.1thm2}
w_{jk} = \ep^{-\frac{d}{2}} e^{-\frac{S_{jk}}{\ep}} \int_{\R^d} e^{-\frac{\varphi(x)}{\ep}}
\bigl(\Psi b^k\bigr)(x) \bigl(\Op_\ep(\hat{q}_\psi)\Psi b^j\bigr) (x)\, dx + O\Bigl(e^{-\frac{S_{jk}}{\ep}} \ep^\infty \Bigr)\; .
\end{equation}

{\sl Step 2:} Next we use an adapted version of stationary phase.

On $G_0$ we choose linear independent
tangent unit vector fields $E_m$, $1\leq m \leq \ell $, and linear independent normal unit vector fields 
$N_m$, $\ell+1\leq m\leq d$, where we set
$N_d = e_d$, the normal vector field on $\Hd_d$. Possibly shrinking $\supp \Psi$, the diffeomorphism $k$ given in \eqref{tubk} can be chosen such that
for each $x\in \supp \Psi$ there exists exactly one $s\in G_0$ and
$t\in (-\delta, \delta)^{d-\ell}$ such that 
\begin{equation}\label{1.2thm2}
x= s + \sum_{m=\ell+1}^d t_{m-\ell} N_m(s) \quad\text{for}\quad k(x) = (s, t) \, . 
\end{equation}
This follows from the proof of the Tubular Neighborhood Theorem, see e.g. \cite{hirsch}. It allows to continue the vector fields $N_m$ from 
$G_0$ to $\supp \Psi$ by setting $N_m(x):= N_m(s)$, thus
$N_m = \partial_{t_{m-\ell}}$. It follows that these vector fields $N_m(x)$ actually satisfy the conditions above 
Hypothesis \ref{hypgeo2a} (in particular, they commute).
We define 
\[ \tilde{\varphi}:= \varphi \circ k^{-1} : G_0\times (-\delta, \delta)^{d-\ell} \rightarrow \R \quad\text{ with }\quad
\tilde{\varphi}(s,t) := \varphi \circ k^{-1}(s,t) = \varphi (x)\; .\]

Since $\varphi(x) = d^j(x) + d^k(x) + C_0 x_d^2 - S_{jk}$ it follows from the construction above that
\begin{align}\label{2.2thm2}
\tilde{\varphi}|_{k(G_0)} &= \varphi|_{G_0} = 0 \\
 E_m \varphi|_{G_0} &= 0\,,\quad \text{for}\;1\leq m \leq \ell \nonumber\\ 
\partial_{t_m}\tilde{\varphi}|_{k(G_0)} &= N_{m+\ell} \varphi|_{G_0} = 0\, , \quad\text{for}\; 1\leq m \leq d-\ell \nonumber\\
D\varphi|_{G_0} &= 0\; .\nonumber
\end{align}

By Hypothesis \ref{hypgeo2a} the transversal Hessian of the restriction of $d^j + d^k$ to $\Hd_d$ at $G_0$ 
is positive definite, i.e.
\begin{equation}\label{1a.2thm2}
D^2_{\perp,G_0} \bigl(d^j + d^k\bigr) = \Bigl( N_m N_{m'} (d^j + d^k)|_{G_0} \Bigr)_{\ell +1\leq m, m' \leq d-1}\, >0\; .
\end{equation}
Analog to the proof of Theorem \ref{wjk-expansion} we use that $d^j+d^k$ is constant along the geodesics. Thus, for any $x_0\in G_0$, the matrix 
$\bigl( N_r N_p (d^j + d^k)(x_0)\bigr)_{\ell + 1 \leq r,p \leq d}$ has $d-\ell - 1$ positive eigenvalues and one zero eigenvalue and in particular 
its determinant is zero. 
Since 
\begin{equation}\label{5.2thm2}
 N_r N_p\varphi = \begin{cases} 
                   2 C_0 + N_d N_d (d^j + d^k) \quad\text{for}\quad (r,p)= (d,d)\\
                   N_r N_p (d^j + d^k)\quad\text{otherwise}\; ,
                  \end{cases}
\end{equation}
the Hessian $\bigl( N_m N_{m'} \varphi|_{G_0} \bigr)_{\ell+1\leq m,m' \leq d}$ of $\varphi$ restricted to $G_0$ is a non-negative quadratic form.
It is in fact positive definite since for any $x_0\in G_0$
\begin{multline}\label{6.2thm2}
 \det \Bigl( N_m N_{m'} \varphi (x_0) \Bigr)_{\ell+1\leq m,m' \leq d} \\[2mm] 
 = \det \Bigl( N_m N_{m'} (d^j + d^k) (x_0) \Bigr)_{\ell+1\leq m,m' \leq d}   
  + \det \begin{pmatrix} \Bigl( N_m N_{m'} (d^j + d^k) (x_0) \Bigr)_{\ell+1\leq m,m' \leq d-1} & 0 \\ * & 2 C_0 \end{pmatrix}\\[2mm]
 = 2 C_0 \,\det D^2_{\perp,G_0}(d^j + d^k)(x_0) > 0 \; .
\end{multline}
Thus 
\begin{equation}\label{6a.2thm2}
D_t^2 \tilde{\varphi}|_{k(G_0)}= \Bigl( N_m N_{m'} \varphi|_{G_0} \Bigr)_{\ell+1\leq m,m' \leq d} >0\; .
\end{equation}

The following lemma is an adapted version of the Morse Lemma with parameter (see e.g. Lemma 1.2.2 in \cite{Dui}). 

\begin{Lem}\label{Morse}
Let $\phi\in\Ce^\infty \bigl(G_0\times (-\delta, \delta)^{d-\ell}\bigr)$ be such that $\phi(s,0) = 0$, $D_t\phi (s,0) = 0$ and the
transversal Hessian $D^2_t\phi(s,\cdot)|_{t=0} =: Q(s)$ is non-degenerate for all $s\in G_0$. Then, for each $s\in G_0$, there is a diffeomorphism
$ y(s,.):  (-\delta, \delta)^{d-\ell} \rightarrow U$, where $U \subset\R^{d-\ell}$ is some neighborhood of $0$, such that 
\begin{equation}\label{3.2thm2} 
 y(s,t) = t + O\bigl(|t|^2\bigr) \quad\text{as}\;\; |t|\to 0\quad\text{and}\quad \phi(s,t) = \frac{1}{2}\langle y(s,t),  Q(s) y(s,t)\rangle\; . 
\end{equation}
Furthermore, $y(s,t)$ is $\Ce^\infty$ in $s\in G_0$.
\end{Lem}

The proof of Lemma \ref{Morse} follows the proof of the Morse-Palais Lemma in \cite{lang}, noting that the construction depends smoothly on the 
parameter $s\in G_0$.

By \eqref{2.2thm2} and \eqref{6a.2thm2}, the phase function $\tilde{\varphi}$ satisfies the assumptions on $\phi$ given in Lemma \ref{Morse}.
We thus can define the diffeomorphism $h:= \id \times y: G_0\times (-\delta, \delta)^{d-\ell}\rightarrow G_0 \times U$ for
$y$ constructed with respect to $\tilde{\varphi}$ as in Lemma \ref{Morse}. 
Using the diffeomorphism $k: \supp \Psi\rightarrow G_0\times (-\delta, \delta)^{d-\ell}$ constructed above (see \eqref{1.2thm2}), 
we set $g(x)= h\circ k (x) = (s, y)$ (then $g^{-1}(s,0) = s$ holds for any $s\in G_0$). Thus
\begin{equation}\label{10.2thm2}
 \varphi \bigl(g^{-1}(s,y)\bigr) = \frac{1}{2} \langle  y, Q(s) y\rangle
\end{equation}
and setting $x=g^{-1}(s,y)$ we obtain
by \eqref{12.1thm2}, 
using the notation \eqref{14thm1}, 
modulo $O\bigl(e^{-\frac{S_{jk}}{\ep}} \ep^\infty \bigr)$
\begin{align}
w_{jk} & \equiv \ep^{-\frac{d}{2}} e^{-\frac{S_{jk}}{\ep}} \int_{\supp \Psi} e^{-\frac{\varphi(x)}{\ep}} a(x;\ep)\, dx \nonumber \\
 & =  \ep^{-\frac{d}{2}} e^{-\frac{S_{jk}}{\ep}} \int_{G_0}\int_{U}
e^{-\frac{1}{2\ep} \langle y, Q(s) y\rangle} a(g^{-1}(s,y); \ep)  J(s,y) \, dy \, d\sigma(s)  \label{4.2thm2}
\end{align}
where $d\sigma$ is the Euclidean surface element on $G_0$ and $J(s,y)= \det D_y g^{-1}(s, .)$ denotes the Jacobi determinant  for the
diffeomorphism 
\[ g^{-1}(s, .): U \rightarrow \Span \bigl(N_{\ell+1}(s), \ldots, N_d(s)\bigr) \] 
and $Q(s)=D^2_t\tilde{\varphi}(s,\cdot)|_{t=0}$ denotes the transversal Hessian of $\tilde{\varphi}$ 
as given in \eqref{6a.2thm2}.  From the construction of $g$ and \eqref{1.2thm2} it follows that $J(s,0) = 1$ for all $s\in G_0$.

By the stationary phase formula with respect to $y$ in \eqref{4.2thm2}, we get modulo $O\bigl(e^{-\frac{S_{jk}}{\ep}} \ep^\infty \bigr)$
\begin{align}
 w_{jk} &= \ep^{-\frac{d}{2}} e^{-\frac{S_{jk}}{\ep}} \bigl(\ep 2 \pi\bigr)^{\frac{d-\ell}{2}} 
\int_{G_0} \bigl(\det Q(s)\bigr)^{-\frac{1}{2}}
\sum_{\nu=0}^\infty \frac{\ep^\nu}{\nu !}\Bigl( \langle \partial_y, Q^{-1}(s) \partial_y\rangle^\nu \tilde{a} J\Bigr)(s,0; \ep)\, d\sigma(s) \nonumber\\
&=  \ep^{-\frac{\ell}{2}} e^{-\frac{S_{jk}}{\ep}}\bigl(2 \pi\bigr)^{\frac{d-\ell}{2}}\sum_{\nu=0}^\infty
\ep^\nu  \int_{G_0}B_{\nu} (s)\, d\sigma(s)\label{9.2thm2}
\end{align}
where $\tilde{a}(.; \ep) := a(.; \ep)\circ g^{-1}$ and, for any $s\in G_0$, $B_{0}(s)$ is given by the leading order of
\begin{equation}
 \Bigl(\det Q(s)\Bigr)^{-\frac{1}{2}} a(g^{-1}(s,0); \ep)\bigr) 
=  \Bigl|2C_0 \det D^2_{\perp,G_0}\bigl(d^j + d^k\bigr)(s)\Bigr|^{-\frac{1}{2}} a(s; \ep)\, ,\qquad \label{7.2thm2}
\end{equation}
using \eqref{6.2thm2}, \eqref{6a.2thm2} and identifying $s\in G_0$ with a point in $\R^d_x$.

We now use the definition of $a$ in \eqref{14thm1}, the expansion \eqref{5thm1} of $\Op_{\ep,0}(\hat{q}_\psi) \Psi b^j$ and the fact that
\eqref{20thm1} and \eqref{21thm1} also hold for any $y_0\in G_0$ in the setting of Hypothesis \ref{hypgeo2} to get for $s\in G_0$ 
\begin{equation}\label{8.2thm2}
 B_{0} (s) 
=\sqrt{\frac{\ep}{2\pi}} \Bigl|\det D^2_{\perp, G_0}\bigl(d^j + d^k\bigr)(s)\Bigr|^{-\frac{1}{2}} 
\ep^{-(N_j + N_k)} b^k_{-N_k}(s) \sum_{\eta\in\Z^d} \tilde{a}_\eta(s)
\eta e^{-\eta\cdot \nabla d^j (s)} b^j_{-N_j} (s) \; .
\end{equation}
Combining \eqref{8.2thm2} and \eqref{9.2thm2} and using \eqref{22thm1} completes the proof.

\qed

\section{Some more results for $w_{jk}$}\label{section5}

In this section, we derive some formulae and estimates for the interaction term $w_{jk}$ and its leading order term, assuming only Hypotheses \ref{hyp1} to \ref{hypkj}, i.e. without any assumptions on the  
geodesics between the potential minima $x^j$ and $x^k$.

We combine the fact that the relevant jumps in the interaction term are those taking place in a small neighborhood of 
$\Hd_d\cap E$, proven in \cite{kleinro4}, Proposition 1.7, with the results on approximate
eigenfunctions proven in \cite{kleinro5}.

\begin{prop}\label{wjkasymp}
Assume that Hypotheses \ref{hyp1} to \ref{hypkj} hold and let $\hat{v}_m^\ep,\, m=j,k,$ denote the 
approximate eigenfunctions given in \eqref{hatvm}. For $\delta>0$, we set 
\begin{equation}\label{deltagammac}
\delta\Gamma := \delta\Hd_{d,R}\cap E\, , \quad \widehat{\delta\Gamma} := \delta \Gamma \cap \Hd_{d,R}
\quad\text{and}\quad
\widehat{\delta\Gamma}^{c} := \delta \Gamma \cap
\Hd_{d,R}^{c}
\end{equation}
where $\delta \Hd_{d,R}$ is defined in \eqref{deltaA}.
Then the interaction term is given by
\begin{equation}\label{0prop51}
w_{jk} =
 \skpd{ \hat{v}_j^\ep}{ \id_{\widehat{\delta\Gamma}}T_\ep \id_{\widehat{\delta\Gamma}^c}
 \hat{v}_k^\ep}
- \skpd{\id_{\widehat{\delta\Gamma}}T_\ep \id_{\widehat{\delta\Gamma}^c}
\hat{v}_j^\ep}{\hat{v}_k^\ep}
+ O\Bigl(\ep^\infty e^{-\frac{S_{jk}}{\ep}}\Bigr)\; .
\end{equation}
Moreover, setting
\begin{equation}\label{tdelta}
\tilde{t}^\delta (x, \xi) := - \sum_{\gamma\in\disk} \id_{\widehat{\delta\Gamma}^c}(x+\gamma) a^{(0)}_\gamma (x) 
\cosh \frac{\gamma\cdot \xi}{\ep} \; , 
\end{equation}
the leading order of $w_{jk}$ is can be written as
\begin{equation}\label{theowjk1}
\sum_{x\in \widehat{\delta\Gamma}_\ep} \hat{v}^\ep_j(x)
\hat{v}^\ep_k(x) \left( \tilde{t}^\delta (x,\nabla d^j(x)) -
\tilde{t}^\delta (x,\nabla d^k(x))\right)  \; .
\end{equation}
If $\hat{v}^\ep_j$ and $\hat{v}^\ep_k$ are both strictly positive in
$\widehat{\delta\Gamma}_\ep$, we have modulo $O\left(\ep^\infty e^{-\frac{S_{jk}}{\ep}}\right)$
\begin{multline}\label{theowjk2}
\sum_{x\in \widehat{\delta\Gamma}_\ep} \hat{v}^\ep_j(x) \hat{v}^\ep_k(x)
\nabla_\xi \tilde{t}^\delta (x,\nabla d^k(x))(\nabla d^j(x) - \nabla d^k(x)) \\
\leq w_{jk} \leq \sum_{x\in \widehat{\delta\Gamma}_\ep} \hat{v}^\ep_j(x)
\hat{v}^\ep_k(x) \nabla_\xi \tilde{t}^\delta (x,\nabla d^j(x))(\nabla
d^j(x) - \nabla d^k(x))\; .
\end{multline}

\end{prop}

We remark that the translation operator $\id_{\widehat{\delta\Gamma}}T_\ep \id_{\widehat{\delta\Gamma}^c}$ is non-zero only for translations 
mapping points $x\in E$ with $0\leq x_d\leq \delta$ to points $x+\gamma\in E$ with $-\delta \leq x+\gamma < 0$. Thus each translation crosses
the hyperplane $\Hd_d$ from right to left. 

\begin{proof}
Since by Hypothesis \ref{hypkj} each of the two wells has
exactly one eigenvalue within the spectral interval $I_\ep$, we
have $\hat{v}_j^\ep:= \tilde{v}_{j,1}^\ep = \hat{v}_{j,1}^\ep$ in the setting of \cite{kleinro5}, Theorem 1.8.
Setting
\begin{equation}\label{defA1} 
A:= \id_{\widehat{\delta\Gamma}}T_\ep \id_{\widehat{\delta\Gamma}^c} - \id_{\widehat{\delta\Gamma}^c}T_\ep \id_{\widehat{\delta\Gamma}}\; ,
\end{equation}
we have by \cite{kleinro5}, Proposition 1.7,
\begin{eqnarray}
\left| w_{jk} -  \skpd{\hat{v}^\ep_j}{A \hat{v}^\ep_k} \right| &=&
\left| \skpd{v_j}{A v_k} -  \skpd{\hat{v}^\ep_j}{A \hat{v}^\ep_k} \right|  +
\expord{-(S_0 + a - \delta)}\nonumber \\
&\leq& \left| \skpd{v_j-\hat{v}^\ep_j}{A v_k}\right| +
 \left|\skpd{\hat{v}^\ep_j}{A (v_k- \hat{v}^\ep_k)} \right| + \expord{-(S_0 + a - \delta)}\; .
 \label{wjkminusbeide}
\end{eqnarray}
From \eqref{defA1} and the triangle inequality for the Finsler distance $d$ it follows that
\begin{multline*}
\left| \skpd{v_j-\hat{v}^\ep_j}{A v_k}\right| =
\Bigl|\sum_{x\in\disk}\sum_{\gamma\in\disk}
\left[\id_{\widehat{\delta\Gamma}}(x)\id_{\widehat{\delta\Gamma}^c}(x+\gamma) - 
\id_{\widehat{\delta\Gamma}^c}(x)\id_{\widehat{\delta\Gamma}}(x+\gamma) \right]\times\\
\times\, e^{\frac{d^j(x)}{\ep}}e^{-\frac{d^j(x)}{\ep}}
\left(v_j(x) - \hat{v}^\ep_j(x)\right) a_\gamma(x) e^{\frac{d^k(x)}{\ep}}e^{-\frac{d^k(x)}{\ep}} v_k(x+\gamma)\Bigr| \\
\leq e^{-\frac{d(x_j, x_k)}{\ep}} \left\| e^{\frac{d^j}{\ep}}(v_j - \hat{v}^\ep_j)
\right\|_{\ell^2(\delta\Gamma)}
\left\| e^{\frac{d^k}{\ep}}v_k
\right\|_{\ell^2(\delta\Gamma)}
\sum_{|\gamma|<B} \left\| a_\gamma e^{\frac{d(.,.+\gamma)}{\ep}}
\right\|_{\ell^\infty(\delta\Gamma)}\; .
\end{multline*}
In the last step we used that for some $B>0$
we have $|\gamma|<B$ if $x\in \widehat{\delta\Gamma}$ and
$x+\gamma\in\widehat{\delta\Gamma}^c$ and vice versa.
Therefore by \cite{kleinro5}, Theorem 1.8, Proposition 3.1 and by
\eqref{agammasupnorm2} we have
\begin{equation}\label{vjminusujinw}
\left| \skpd{v_j-\hat{v}^\ep_j}{A v_k}\right| =
O\left(e^{-\frac{S_{jk}}{\ep}} \ep^\infty\right)\; .
\end{equation}
The second summand on the right hand side of \eqref{wjkminusbeide}
can be estimated similarly. This proves \eqref{0prop51}.\\ 

For the next step, we remark that by Hypothesis \ref{hyp1}, as a function on the cotangent bundle
$T^*\delta \Gamma$,  the symbol $\tilde{t}^\delta$ is
hyperregular (see \cite{kleinro}).

Setting $\tilde{b}^\ell:= b^\ell_{-N_\ell}$ for $\ell\in\{j,k\}$, \eqref{0prop51} leads to
\begin{multline}\label{wjk11}
w_{jk} \equiv \sum_{x\in \widehat{\delta\Gamma}_\ep}\sum_{\natop{\gamma\in\disk}{x+\gamma \in \widehat{\delta\Gamma}_\ep^c}} a^{(0)}_\gamma(x)
\ep^{\frac{d}{2}-N_j - N_k} \left( \tilde{b}^j(x) e^{-\frac{d^{j}(x)}{\ep}} \tilde{b}^k(x+\gamma) e^{-\frac{d^{k}(x+\gamma)}{\ep}}\right.\\
 \left. - \tilde{b}^j(x+\gamma) e^{-\frac{d^{j}(x+\gamma)}{\ep}}
\tilde{b}^k(x) e^{-\frac{d^{k}(x)}{\ep}}\right)\; .
\end{multline}
We split the sum over $\gamma$ in the parts $A_1(x)$ with $|\gamma|\leq 1$ and $A_2(x)$ with $|\gamma|>1$. 
Then it follows at once from \eqref{agammasum} 
that for any $B>0$ and some $C>0$
\begin{equation}\label{wjk17}
 \Big| \sum_{x\in \widehat{\delta\Gamma}_\ep}A_2(x)\Bigr| \leq C e^{-\frac{B}{\ep}}\; .
\end{equation}
To analyze $A_1(x)$, we use Taylor expansion at $x$,  yielding for $\ell=j,k$
\begin{equation}\label{wjk15}
\sum_{\natop{\gamma \in \disk}{|\gamma|\leq 1}}  \id_{\widehat{\delta\Gamma}^c}(x+\gamma) a^{(0)}_\gamma(x)  \tilde{b}^\ell (x+\gamma)
e^{-\frac{d^{\ell}(x+\gamma)}{\ep}} = - \tilde{b}^\ell(x)
e^{-\frac{1}{\ep}d^\ell(x)} \tilde{t}^\delta(x, \nabla d^\ell(x)) + R_1(x)
\end{equation}
where, using the notation $\gamma = \ep \eta$ for $\eta\in\Z^d$ and $\tilde{a}_\eta = a^{(0)}_{\ep\eta}$, the remainder $R_1(x)$ can for
some $C>0$ and any $B>0$ be estimated by
\begin{align}\label{wjk16}
\bigl| R_1(x) \bigr| &= e^{-\frac{d^{\ell}(x)}{\ep}} \Bigl| \sum_{\natop{\eta \in \Z^d}{|\eta|\leq \frac{1}{\ep}}} 
\id_{\widehat{\delta\Gamma}^c}(x+\ep\eta) \ep \eta\cdot \nabla  \tilde{b}^\ell (x) e^{\eta\nabla d^\ell(x)} \tilde{a}_\eta(x)  (1 + O(1))\Bigr| \\
&\leq \ep C \sum_{\eta\in\Z^d} |\eta| e^{-B|\eta|} \leq \ep C \int_{\R^d}  |\eta| e^{-B|\eta|}\, d\eta \leq C \ep \; .
\end{align}
Inserting \eqref{wjk17}, \eqref{wjk15} and \eqref{wjk16} into \eqref{wjk11} yields 
\[
w_{jk} \equiv \sum_{x\in\widehat{\delta\Gamma}_\ep}
\ep^{\frac{d}{2}-N_j-N_k} \tilde{b}^j(x)\tilde{b}^k (x) e^{-\frac{1}{\ep}(d^j(x) +
d^k(x))} \left( \tilde{t}^\delta(x, \nabla d^j(x)) -
\tilde{t}^\delta(x, \nabla d^k(x))\right) + O(\ep)\, .
\]
and thus proves \eqref{theowjk1}.

To show \eqref{theowjk2}, we use that for any convex function $f$
on $\R^d$
\[ \nabla f(\eta)(\xi -\eta) \leq f(\xi) - f(\eta) \leq \nabla f(\xi) (\xi - \eta)\, ,
\qquad \eta,\xi\in\R^d\; .\]
Thus for $\hat{v}^\ep_j$ and $\hat{v}^\ep_k$ both strictly positive in
$\widehat{\delta\Gamma}$,
\eqref{theowjk2} follows from the convexity of $\tilde{t}^\delta$.

\end{proof}

\begin{appendix}

\section{Pseudo-Differential operators in the discrete setting}\label{app1}

We introduce and analyze pseudo-differential operators associated to symbols, which are $2\pi$-periodic with respect to $\xi$ 
(for former results see also \cite{kleinro2}).

Let $\T^d := \R^d/(2\pi)\Z^d$ denote the $d$-dimensional torus and without further mentioning we identify functions 
on $\T^d$ with $2\pi$-periodic functions on $\R^d$.

\begin{Def}\label{pseudo}
\ben
\item An order function on $\R^N$ is a function $m:\R^{N} \ra (0, \infty)$ such that there exist $C>0, M\in \N$ such that
\[ m(z_1) \leq C \langle z_1-z_2\rangle^M m(z_2) \, , \qquad z_1, z_2\in \R^{N} \]
where $\langle x \rangle := \sqrt{1+|x|^2}$.
\item A function $p\in \Ce^\infty \bigl(\R^{N}\times (0, 1]\bigr)$ is an element of the symbol class 
$S_\delta^k\bigl(m\bigr)\bigl(\R^{N}\bigr)$ for some order function $m$ on $\R^N$, if for all $\alpha\in \N^{N}$ 
there is a constant $C_\alpha >0$ such that
\[ \Bigl| \partial^\alpha p (z; \ep)\Bigr| \leq C_\alpha \ep^{k-\delta |\alpha|} m(z)\, , \qquad z\in\R^N \]
uniformly for $\ep\in (0,1]$. 
On $S_\delta^k\bigl(m\bigr)\bigl(\R^{N}\bigr)$ we define the Fr\'echet-seminorms 
\begin{equation}\label{Frechet-Norm}
\|p\|_{\alpha} := \sup_{z\in\R^N, 0<\ep\leq 1}\frac{\Bigl| \partial^\alpha p (z; \ep)\Bigr|}{\ep^{k-\delta |\alpha|} m(z)}\, , \quad \alpha \in\N^N\; .
\end{equation}
We define the symbol class $S_\delta^k\bigl(m\bigr)\bigl(\R^{N}\times \T^d\bigr)$ by identification of 
$\Ce^\infty (\T^d)$ with the $2\pi$-periodic functions in $\Ce^\infty (\R^d)$. 
\item 
To $p\in S_\delta^k\bigl(m\bigr)\bigl(\R^{2d}\times \T^d\bigr)$ we associate a pseudo-differential operator 
$\widetilde{\Op}_\ep^{\T}(p): {\mathcal K}\left(\disk\right) \longrightarrow
{\mathcal K}'\left(\disk\right)$ setting
\begin{equation}\label{psdo3dTorus}
\widetilde{\Op}_\ep^{\T}(p)\, v(x; \ep) := (2\pi)^{-d} \sum_{y\in\disk}\int_{[-\pi,\pi]^d}
e^{\frac{i}{\ep}(y-x)\xi}
p(x, y ,\xi;\ep)v(y) \, d\xi 
\end{equation}
where
\begin{equation}\label{kompaktge}
{\mathcal K}\left(\disk\right):=\{ u: \disk\rightarrow \C\; |\; u~\mbox{has compact support}\}
\end{equation}
and ${\mathcal K}'\left(\disk\right):= \{f: \disk\rightarrow \C\ \} $ is dual to
${\mathcal K}\left(\disk\right)$
by use of the scalar product $\skpd{u}{v}:= \sum_x \bar{u}(x)v(x)$ .
\item 
For $t\in [0,1]$ and $q\in S_\delta^k\bigl(m\bigr)\bigl(\R^{d}\times \T^d\bigr)$ the associated 
pseudo-differential operator $\Op_{\ep,t}^\T (q)$ is defined by
\begin{equation}\label{psdo2dTorus}
\Op_{\ep,t}^{\T}(q)\, v(x; \ep) := (2\pi)^{-d} \sum_{y\in\disk}\int_{[-\pi,\pi]^d}
e^{\frac{i}{\ep}(y-x)\xi}
q((1-t)x + ty , \xi; \ep) v(y) \, d\xi
\end{equation}
for any $v\in\mathcal{K}(\disk)$ and we set $\Op_{\ep, 0}^\T (q) =: \Op_\ep^\T(q)$.
\item 
To $p\in S_\delta^k\bigl(m\bigr)\bigl(\R^{3d}\bigr)$ we associate a pseudo-differential operator 
$\widetilde{\Op}_\ep (p): \Ce_0^\infty (\R^d) \longrightarrow \mathcal{D}'(\R^d)$ setting
\begin{equation}\label{psdo3d}
\widetilde{\Op}_\ep (p)\, v(x; \ep) := (2\pi\ep )^{-d} \int_{\R^{2d}}
e^{\frac{i}{\ep}(y-x)\xi}
p(x, y ,\xi;\ep)v(y) \, dy \, d\xi \; .
\end{equation}
\item 
For $t\in [0,1]$ and $q\in S_\delta^k\bigl(m\bigr)\bigl(\R^{2d}\bigr)$ the associated a pseudo-differential 
operator $\Op_{\ep,t} (q)$ is defined by
\begin{equation}\label{psdo2dt}
\Op_{\ep,t} (q)\, v(x; \ep) := (2\pi\ep )^{-d} \int_{\R^{2d}}
e^{\frac{i}{\ep}(y-x)\xi}
q((1-t)x + ty , \xi; \ep) v(y) \, dy \, d\xi \, , \qquad v\in\Ce_0^\infty{\R^d}
\end{equation}
and we set $\Op_{\ep,0}(q) =: \Op_\ep(q)$.
\een
\end{Def}

Standard arguments show that $\widetilde{\Op}_\ep(p)$ actually maps $\Ce_0^\infty(\R^d)$ into $\Ce^\infty(\R^d)$. 
Moreover, the seminorms given in \eqref{Frechet-Norm} induce the
structure of a Fr\'echet-space in $S^k_\delta(m)(\R^N)$.

In \cite{kleinro2} we discussed properties of pseudo-differential operators $\Op_\ep^\T(.)$.
In particular we showed that, 
for a symbol $q\in S_\delta^k\bigl(m\bigr)\bigl(\R^{2d}\bigr)$ which is $2\pi$-periodic with respect to $\xi$, 
the restriction of $\Op_\ep (q)$ to $\mathcal{K}\bigl(\disk\bigr)$ coincides with $\Op_\ep^\T (q)$. 

In the next proposition we show that this statement also holds in the more general case of $\widetilde{\Op}_\ep^\T$ and
$\widetilde{\Op}_\ep$.

\begin{prop}\label{prop1app}
For some order function $m$ on $\R^{3d}$, let $p\in S^k_\delta \bigl(m\bigr)\bigl(\R^{3d}\bigr)$ satisfy 
$p(x, y, \xi; \ep) = p(x, y, \xi + 2\pi\eta; \ep)$ for any $\eta\in\Z^d, \xi, x, y\in \R^d$ 
and $\ep\in (0,1]$. Then $p\in S^k_\delta \bigl(m\bigr)\bigl(\R^{2d}\times \T^d\bigr)$ and using the restriction map
\begin{equation}\label{restrict} 
r_\ep: \Ce_0^\infty(\R^d) \ra \mathcal{K}(\disk)\, , \quad r_\ep(u) = u|_{\disk} 
\end{equation}
we have
\begin{equation}\label{0prop1app}
r_\ep \circ \widetilde{\Op}_{\ep} (p)\, u = \widetilde{\Op}_\ep^\T (p) r_\ep\circ u \, , \qquad u\in \Ce_0^\infty(\R^d)\; .
\end{equation}
\end{prop}

\begin{proof}
For $x\neq \disk$ both sides of \eqref{0prop1app} are zero, so we choose $x\in\disk$. Then for $u\in \Ce_0^\infty(\R^d)$, 
using the $\ep$-scaled Fourier transform
\begin{equation}\label{fourier}
F_\ep u (x) =\sqrt{2\pi}^{-d} \int_{\R^d} e^{-\frac{i}{\ep}x\xi} u(\xi) \, d\xi\; ,
\end{equation}
we can write
\begin{equation}\label{1prop1app}
\widetilde{\Op}_\ep (p) u (x; \ep) = \bigl(\ep\sqrt{2\pi}\bigr)^{-d} \int_{\R^d} \bigl( F_\ep p (x, y, \cdot, \ep)\bigr)(x-y) u(y)\, dy\; .
\end{equation}
Since for any $2\pi$-periodic function $g\in\Ce^\infty (\R^d)$ the Fourier transform is given by
\begin{equation}\label{opall1}
F_\ep g = \left(\frac{\ep}{\sqrt{2\pi}}\right)^d \sum_{z\in\disk} \delta_z c_z\, , \quad\text{where}\quad
c_z:= \int_{[-\pi,\pi]^d} e^{-\frac{i}{\ep}z\mu} g(\mu)\, d\mu \; ,
\end{equation}
 (see e.g. \cite{hormander2}), we formally get
\begin{align}
\text{rhs} \eqref{1prop1app} &=  \bigl(\ep\sqrt{2\pi}\bigr)^{-d} \int_{\R^d} \left(\frac{\ep}{\sqrt{2\pi}}\right)^d
\sum_{z\in\disk} \int_{[-\pi, \pi]^d} e^{-\frac{i}{\ep}z\mu} p (x, y, \mu ; \ep)\, d\mu \delta_z (x-y) u(y)\, dy\nonumber\\
&= \bigl(2\pi\bigr)^{-d} \sum_{z\in\disk} \int_{[-\pi, \pi]^d} \int_{\R^d} 
e^{-\frac{i}{\ep}z\mu} p (x, y, \mu ; \ep) \delta_z (x-y) u(y)\, dy\, d\mu\nonumber\\
&=  \bigl(2\pi\bigr)^{-d} \sum_{z\in\disk} \int_{[-\pi, \pi]^d} 
e^{-\frac{i}{\ep}z\mu} p (x, x-z, \mu ; \ep) u(x-z)\, d\mu\; .\label{2prop1app}
\end{align}
With the substitution $y=x-z$ and $\xi = \mu$ we get by \eqref{1prop1app} and \eqref{2prop1app}
\[ \widetilde{\Op}_\ep (p) u(x; \ep) = 
 \bigl(2\pi\bigr)^{-d} \sum_{y\in\disk} \int_{[-\pi, \pi]^d} 
e^{-\frac{i}{\ep}(x-y)\xi} p (x, y, \xi ; \ep) u(y)\, dy\, d\xi = \widetilde{\Op}_\ep^\T (p) u(x; \ep) \]
proving the stated result.

\end{proof}

\begin{rem}\label{remprop1app}
Let $m$ be an order function on $\R^{2d}$ and $p\in S^k_\delta\bigl(m\bigr)\bigl(\R^{2d}\bigr)$ a symbol. 
Then, setting $\tilde{p}_t (x, y, \xi; \ep) := p( tx + (1-t) y, \xi; \ep)$ for $t\in[0,1]$, 
we have $\widetilde{\Op}_\ep (\tilde{p}_t) = \Op_{\ep, t} (p)$. Thus the $t$-quantization can be seen as a 
special case of the general quantization. 

Moreover,  if $p$ is periodic in $\xi$, i.e. if
$p(x, \xi; \ep) = p(x, \xi + 2\pi\eta; \ep)$ for any $\eta\in\Z^d, \xi,x\in \R^d$ and $\ep\in (0,\ep_0]$, then
$p\in S^k_\delta\bigl(m\bigr)\bigl(\R^{d}\times \T^d\bigr)$, 
\[
r_\ep \circ \Op_{\ep,t} (p) (u) = \Op_{\ep,t}^\T (p) \circ r_\ep (u) \, , \qquad u\in \Ce_0^\infty(\R^d)
\]
and $\widetilde{\Op}_\ep^\T (\tilde{p}_t) = \Op_{\ep,t}^\T (p)$. 
\end{rem}

\begin{rem}
For $a\in S_\delta^k\bigl(\langle \xi\rangle^\ell, \R^{3d}\bigr)$ the operator $\widetilde{\Op}_\ep (a)$ is continuous: 
$\mathscr{S}(\R^d) \rightarrow \mathscr{S}(\R^d)$ (see e.g. \cite{martinez}) and, similar to Lemma A.2 in
\cite{kleinro2}, this result implies that $\widetilde{\Op}_\ep^\T(a)$ is continuous:
$s(\disk) \rightarrow s(\disk)$ by use of Proposition \ref{prop1app}.
\end{rem}

The following proposition gives a relation between the different quantizations for symbols which are periodic with respect to $\xi$. The proof is partly based on 
\cite{martinez}, where the result is shown for symbols in $S_0^0\bigl(\langle \xi\rangle^m\bigr)\bigl(\R^{3d}\bigr)$.

\begin{prop}\label{prop2app}
For $0\leq \delta < \frac{1}{2}$, let $a\in S_\delta^k\bigl(m\bigr)\bigl(\R^{2d}\times \T^d\bigr)$ and $t\in [0,1]$, then there exists a unique symbol 
$a_t\in S_\delta^k\bigl(\tilde{m}\bigr)\bigl(\R^d\times \T^d\bigr)$ where $\tilde{m}(x,\xi) := m(x, x, \xi)$ such that
\begin{equation}\label{0prop2app}
\widetilde{\Op}_\ep^\T (a) = \Op_{\ep, t}^\T (a_t)\; . 
\end{equation}
Moreover the mapping $S^k_\delta(m)\ni a\mapsto a_t\in S^k_\delta(\tilde{m})$ is continuous in its Fr\'echet-topology induced from 
\eqref{Frechet-Norm}. $a_t$ can be written as
\begin{equation}\label{1prop2app}
a_t (x, \xi; \ep) 
= (2\pi)^{-d} \sum_{\theta\in\disk} \int_{[-\pi,\pi]^d} e^{\frac{i}{\ep}(\xi - \mu) \theta} a\bigl(x+t\theta, x-(1-t)\theta, \mu; \ep\bigr)\, d\mu
\end{equation}
and has the asymptotic expansion
\begin{equation}\label{2prop2app}
a_t (x, \xi; \ep) \sim \sum_{j=0}^\infty \ep^j a_{t,j}(x, \xi) \, , \qquad a_{t,j}(x, \xi) := \sum_{\natop{\alpha\in\N^d}{|\alpha|=j} } 
\frac{i^{j}}{\alpha!} \partial_\xi^\alpha
\partial_z^\alpha a \bigl(x + t z, x - (1-t) z, \xi; \ep\bigr)|_{z=0}\; .
\end{equation}
If we write $a_t(x, \xi; \ep)  = \sum_{j\leq N-1} \ep^j a_{t,j}(x, \xi) + S_N(a) (x, \xi; \ep)$
then $S_N(a)\in S_\delta^{k + N(1-2\delta)}\bigl(\tilde{m}\bigr)\bigl(\R^d\times \T^d\bigr)$ and the Fr\'echet-seminorms of $S_N$ depend (linearly) 
on finitely many $\|a\|_{\alpha}$ with $|\alpha| \geq N$.
\end{prop}

\begin{proof}
To satisfy \eqref{0prop2app}, the symbol $a_t$ above has to satisfy in $\mathcal{D}'(\R^{2d})$
\begin{equation}\label{3prop2app}
\int_{[-\pi, \pi]^d} e^{\frac{i}{\ep}(y-x)\mu} a(x, y, \mu; \ep) \, d\mu = \int_{[-\pi, \pi]^d} e^{\frac{i}{\ep}(y-x)\mu} a_t((1-t) x + t y, \mu; \ep) \, d\mu\, .
\end{equation}
Setting $\theta= x-y$ and $z=(1-t)x + ty = x - t\theta$ in \eqref{3prop2app} gives
\begin{equation}\label{4prop2app}
\int_{[-\pi, \pi]^d} e^{-\frac{i}{\ep}\theta\mu} a\bigl(z + t \theta, z - (1-t)\theta, \mu; \ep\bigr) \, d\mu = 
\sqrt{2\pi}^{d}\bigl(\mathscr{F}_\ep a_t (z, \,.\, ; \ep)\bigr) (\theta)
\end{equation}
where
${\mathscr F}_\ep:L^2\left(\T^d\right)
\to \ell^2\left((\ep{\mathbb Z})^d\right)$ denotes the discrete Fourier transform defined by
\begin{equation}\label{Fou}
\bigl({\mathscr F}_\ep f\bigr)(\theta) := \frac{1}{\sqrt{2\pi}^d}
\int_{[-\pi,\pi]^d} e^{-\frac{i}{\ep}\theta \mu}f(\mu)\,d\mu \, ,
\qquad f\in L^2(\T^d)\, , \; \theta\in\disk
\end{equation}
with inverse ${\mathscr F}_\ep^{-1}:\ell^2\left((\ep{\mathbb Z})^d\right)\to L^2\left(\T^d\right)$,
\begin{equation} \label{Fou-1}
\bigl({\mathscr F}_\ep^{-1}v\bigr)(\xi) :=
\frac{1}{\sqrt{2\pi}^d}\sum_{\theta\in\disk} e^{\frac{i}{\ep}\theta \xi}v(\theta),
\qquad v\in \ell^2\left(\disk\right)\, , \; \xi\in\T^d
\end{equation}
where the sum in understood in standard L.I.M-sense.
Thus taking the inverse Fourier transform $\mathscr{F}_\ep^{-1}$ on both sides of \eqref{4prop2app} yields \eqref{1prop2app}.\\
To analyze the right hand side of  \eqref{1prop2app}, we set $\eta=\mu -\xi$ and introduce a cut-off-function
$\zeta\in {\mathcal K}\left(\disk, [0,1]\right)$ with $\zeta = 1$ in a neighborhood of $0$ to get
\begin{align}
a_t (x, \xi; \ep) &= b_{t,1}(x, \xi; \ep) + b_{t,2}(x, \xi; \ep)\qquad\text{with}\label{5prop2app}\\
b_{t,1}(x, \xi; \ep)&:= (2\pi)^{-d} \sum_{\theta\in\disk} \int_{[-\pi,\pi]^d} e^{-\frac{i}{\ep}\eta \theta}(1-\zeta(\theta)) 
a\bigl(x+t\theta, x-(1-t)\theta, \xi + \eta; \ep\bigr)\, d\eta\nonumber\\
b_{t,2}(x, \xi; \ep)&:= (2\pi)^{-d} \sum_{\theta\in\disk} \int_{[-\pi,\pi]^d} e^{-\frac{i}{\ep}\eta \theta}\zeta(\theta)
a\bigl(x+t\theta, x-(1-t)\theta, \xi + \eta; \ep\bigr)\, d\eta\nonumber
\end{align}
The aim is now to show $b_{t,1}\in S^\infty (\tilde{m})(\R^d\times \T^d)$ and $b_{t,2}\in S_\delta^k(\tilde{m})(\R^d\times \T^d)$ having the required
asymptotic expansion and that the mappings $a\mapsto b_{t,1}$ and $a\mapsto b_{t,2}$ are continuous.

Since $e^{\frac{i}{\ep}2\pi \eta z} = 1$ for all $z\in\disk$ and $\eta\in\Z^d$, it follows at once from
\eqref{1prop2app} that $b_{t,i} (x, \xi + 2\pi \eta; \ep) = b_{t,i}(x, \xi; \ep)$ for $i=1,2$.

By use of the operator $L(\theta, \eta):= \frac{-\ep^2\Delta_\eta}{|\theta|^2}$, which is well defined on the support 
of $1-\zeta (\theta)$ and fulfills $L(\theta,\eta)e^{-\frac{i}{\ep}\theta\eta}= e^{-\frac{i}{\ep}\theta\eta}$, we have for
any $n\in \N$ by partial
integration, using the $2\pi$-periodicity of the symbol $a(x, y, \xi; \ep)$ with respect to $\xi$,
\begin{align}
b_{t,1}(x,\xi;\ep) = (2\pi)^{-d} \sum_{\theta\in\disk}\int_{[-\pi,\pi]^d} \left( L^n(\theta,\eta) 
e^{-\frac{i}{\ep}\theta\eta}\right)
                    (\id - \zeta(\theta)) a\bigl(x + t \theta, x - (1-t)\theta,\xi+\eta;\ep\bigr)\, d\eta\nonumber  \\
 = (2\pi)^{-d}\sum_{\theta\in\disk}\int_{[-\pi,\pi]^d} e^{-\frac{i}{\ep}\theta\eta}
             \frac{(\id - \zeta(\theta))}{|\theta|^{2n}}(-\ep^2\Delta_\eta)^n a\bigl(x+ t\theta, x- (1-t)\theta,\xi+\eta;\ep\bigr)\, d\eta \, .\label{6prop2app}
\end{align}
Since $a\in S_\delta^k(m)$, the absolute value of the integrand is for some $C>0$ and $M\in\N$ bounded from above by
\begin{equation}\label{7prop2app}
  C \ep^{k + 2n(1-\delta)}\frac{m(x+ t\theta, x- (1-t)\theta, \xi+\eta)}{\langle \theta \rangle^{2n}} \leq
           C \ep^{k + 2n(1-\delta)}\langle\theta\rangle^{M - 2n}\langle\eta\rangle^M m(x, x, \xi) \, .
\end{equation}
This term is integrable and summable for $n$ sufficiently large yielding
\[ b_{t,1}(x,\xi;\ep) = \ep^{k + 2n(1- \delta)-d}O(\tilde{m}(x,\xi))\; .\]
The derivatives can be
estimated similarly, and thus $ b_{t,1}\in S^\infty(\tilde{m})\bigl(\R^d\times \T^d\bigr)$.\\
To see the continuity of $S^k_\delta(m)\ni a\mapsto b_{t,1}\in S^{k + 2n(1- \delta)-d}_\delta(\tilde{m})$ for any $n\in\N$ large enough, 
we use \eqref{6prop2app} and \eqref{7prop2app} 
to estimate for any $\alpha, \beta\in \N^d$ and $x\in\R^d, \xi\in\T^d$
\begin{align*}
\Bigl|\partial_x^\alpha\partial_\xi^\beta b_{t,1}(x,\xi;\ep)\Bigr| &\leq C \sum_{\theta\in\disk}\int_{[-\pi,\pi]^d} 
             \frac{|\id - \zeta(\theta)|}{|\theta|^{2n}}\frac{\bigl|(-\ep^2\Delta_\eta)^n \partial_x^\alpha\partial_\xi^\beta a(x+ t\theta, x- (1-t)\theta,\xi+\eta;\ep)\bigr|}
             {\ep^{k + 2n(1- \delta)-d}m(x+ t\theta, x- (1-t)\theta,\xi+\eta)}\\
             &\qquad \times \ep^{k + 2n(1- \delta)-d}m(x+ t\theta, x- (1-t)\theta,\xi+\eta)\, d\eta\\
           &\leq C \ep^{k + 2n(1- \delta)-d} m(x,x,\xi)\|a\|_{(\alpha,\tilde{\beta}(n))}  \sum_{\theta\in\disk}\int_{[-\pi,\pi]^d} 
             |\id - \zeta(\theta)|\langle\theta\rangle^{M - 2n}\langle\eta\rangle^M\, d\eta\\
&\leq C \ep^{k + 2n(1- \delta)-d}\tilde{m}(x,\xi)\|a\|_{(\alpha,\tilde{\beta}(n))}
\end{align*}
for $\tilde{\beta}(n) = \beta + (2n,\ldots 2n)$, where the last estimate holds for $n$ sufficiently large. This gives continuity.

Since in the definition of $b_{t,2}$ in \eqref{5prop2app} integral and sum range over a compact set, it follows analog to the estimates above that
\[ 
\Bigl|\partial_x^\alpha \partial_\xi^\beta b_{t,2}(x, \xi; \ep)\Bigr| \leq C_{\alpha,\beta} \ep^{k - (|\alpha| + |\beta|)\delta}\|a\|_{(\alpha, \beta)}
m(x, x, \xi)  \]
and thus $b_{t,2}\in S_\delta^k(\tilde{m})$ and the mapping $S^k_\delta(m)\ni a\mapsto b_{t,2}\in S^k_\delta(\tilde{m})$ is continuous. 

Thus $S_\delta^k(m)\ni a\mapsto a_t\in S_\delta^k(\tilde{m})$ is continuous. Using standard arguments, the method of stationary phase 
(see e.g. \cite{thesis}, Lemma B.4) gives the asymptotic expansion \eqref{2prop2app}. 

Since $a\mapsto S_N = a_t - \sum_{j=0}^{N-1}\ep^j a_{t,j}$ is obviously continuous, each Fr\'echet-seminorm of $S_N$ can be estimated 
by finitely many Fr\'echet-seminorms of $a$. 
To get the more refined statement $S_N$, we use \eqref{1prop2app} to write
\begin{equation}\label{8prop2app}
a_t (x, \xi; \ep) = e^{i\ep D_\theta D_\eta} a\bigl(x + t\theta, x - (1-t)\theta; \xi + \eta; \ep\bigr)\bigl|_{\theta = 0 = \eta}\; . 
\end{equation}
In fact, by algebraic substitutions, \eqref{8prop2app} is a consequence of the formula
\begin{equation}\label{9prop2app}
 e^{i\ep D_\theta D_\eta} b(\theta, \eta; \ep) = \sum_{z\in\disk} \int_{[-\pi, \pi]^d} e^{-\frac{i}{\ep}z\mu} b(\theta - z, \eta - \mu; \ep)\, d\mu
\end{equation}
for $b\in S^k_\delta(\tilde{m})(\R^d\times \T^d)$, where, for $x, \xi$ fixed, we set 
$b(\theta, \eta; \ep) = a(x + t\theta, x - (1-t)\theta; \xi + \eta; \ep)$.
\eqref{9prop2app} may be proved by writing $e^{i\ep D_\theta D_\eta}$ as a multiplication operator in the covariables and applying the 
Fouriertransforms ${\mathscr F}_\ep$, ${\mathscr F}_\ep^{-1}$, using that $e^{-\frac{i}{\ep}x\xi}$ is invariant under 
${\mathscr F}_{\ep, \xi\rightarrow z}{\mathscr F}_{\ep, x\rightarrow \mu}^{-1}$ and the standard fact that Fourier transform maps 
products to convolutions (see \cite{thesis}).

Using Taylor's formula for $e^{ix}$, we get
\[ S_N(a)(x, \xi; \ep) = \frac{(i\ep D_\theta D_\eta)^N}{(N-1)!} \int_0^1 (1-s)^{N-1}e^{i\ep s D_\theta D_\eta}\, ds\;
a \bigl(x + t \theta, x - (1-t)\theta, \eta + \xi; \ep\bigr) |_{\theta = 0 = \eta}\, , \]
proving that $S_N$ only depends on Fr\'echet-seminorms of
$(D_\theta D_\eta)^N a\bigl(x+t\theta, x-(1-t)\theta, \eta + \xi; \ep\bigr)$ and thus not on Fr\'echet-seminorms $\|a\|_\alpha$ with $|\alpha|< N$.

\end{proof}

The norm estimate \cite{kleinro2}, Proposition A.6, for operators $\Op_\ep^\T(q)$ with a bounded symbol 
$q\in S^k_\delta (1)(\R^d\times \T^d)$ combined with Proposition \ref{prop2app}
leads at once to the following corollary.

\begin{cor}\label{cor1app}
Let $a\in S^k_\delta(1)\left(\R^{2d}\times\T^d\right)$
with  $0\leq \delta \leq \frac{1}{2}$.
Then there exists a constant $M>0$ such that, for the associated operator $\widetilde{\Op}_\ep^{\T}(a)$ given by
\eqref{psdo3dTorus} the
estimate
\begin{equation}\label{cvab}
\left\| \widetilde{\Op}_\ep^{\T}(a) u \right\|_{\ell^2(\disk)} \leq M \ep^r \|u\|_{\ell^2(\disk)}
\end{equation}
holds for any $u\in s\left(\disk\right)$ and $\ep>0$.
$\widetilde{\Op}_\ep^{\T}(a)$ can therefore be extended to a continuous operator:
$\ell^2\left(\disk\right)\longrightarrow \ell^2\left(\disk\right)$ with
$\|\widetilde{\Op}_\ep^{\T}(a)\| \leq M\ep^r$. Moreover $M$ can be chosen depending only on a finite number of
Fr\'echet-seminorms of the symbol $a$.
\end{cor}

In the next proposition, we analyze the symbol of an operator conjugated with a term $e^{\varphi/\ep}$.

\begin{prop}\label{prop3app}
Let $q\in S^k_\delta \bigl(1\bigr)\bigl(\R^{2d}\times \T^d\bigr)$, $0\leq \delta<\frac{1}{2}$, be a symbol such that
the map $\xi \mapsto q(x, y, \xi; \ep)$ can be extended to an analytic function on $\C^d$.
Let $\psi\in \Ce^\infty (\R^d, \R)$ such that all derivatives are bounded.\\
Then
\[ Q_\psi:= e^{\psi/\ep} \widetilde{\Op}^\T_{\ep}(q) e^{-\psi / \ep} \] 
is the quantization of the symbol $\hat{q}_\psi\in S^k_\delta \bigl( 1\bigr)\bigl(\R^{2d}\times \T^d\bigr)$ given by
\begin{equation}\label{00prop3app}
\hat{q}_\psi (x, y, \xi; \ep) := q( x, y, \xi - i \Phi (x,y); \ep) 
\end{equation}
where $\Phi$ is given in \eqref{2prop3app}. In particular, the map $\xi \mapsto \hat{q}_\psi(x,y,\xi;\ep)$ can be extended to
an analytic function on $\C^d$. If $q$ has an asymptotic expansion $q\sim \sum_n \ep^n q_n$ in $\ep$, then
the same is true for $\hat{q}_\psi$.

For $t\in [0,1]$, the operator $Q_\psi$ 
is the $t$-quantization of a symbol $q_{\psi, t} \in  S^k_\delta \bigl(1\bigr)\bigl(\R^{d}\times \T^d\bigr)$ with
asymptotic expansion $q_{\psi, t} \sim \sum_n q_{n,\psi,t}$ such that 
$q_{\psi,t} - \sum_{n=0}^{N-1} q_{n,\psi, t} \in S^{k+N(1-2\delta)}(1)(\R^d\times \T^d)$. 
Moreover, the map $\xi \mapsto q_{\psi, t}(x,\xi; \ep)$ can be extended to an analytic function on $\C^d$ and 
\begin{equation}\label{0prop3app}
q_{\psi, t} (x, \xi; \ep) = \hat{q}_\psi (x,x,\xi;\ep) = q ( x, x, \xi - i \nabla \psi (x); \ep) \quad 
\mod  S^{k+1-2\delta}_\delta \bigl(1\bigr)\bigl(\R^{d}\times \T^d\bigr)\; .
\end{equation}
\end{prop}

\begin{proof}
The integral kernel of $e^{\psi/\ep} \widetilde{\Op}^\T_{\ep}(q) e^{-\psi / \ep} $ is given by the 
oscillating integral
\begin{align}
K_\psi (x,y) & := (2\pi)^{-d} \int_{[-\pi, \pi]^d} e^{\frac{i}{\ep}[ (y-x)\xi + i (\psi (y) - \psi (x))]} 
q( x, y, \xi; \ep) \, d\xi \nonumber\\
&=  (2\pi)^{-d} \int_{[-\pi, \pi]^d} e^{\frac{i}{\ep}(y-x) [\xi + i \Phi (x, y)]} q( x, y, \xi; \ep) \, d\xi \label{1prop3app}
\end{align}
where we set 
\begin{equation}\label{2prop3app} 
\Phi (x, y) := \int_0^1 \nabla \psi ( (1-t) x + t y)\, dt\; .
\end{equation} 
Substituting $\tilde{\xi} := \xi + i \Phi (x, y)$ and iteratively using Lemma \ref{Lem1} yields
\begin{align}
\text{rhs }\eqref{1prop3app} &= 
(2\pi)^{-d} \int_{[-\pi, \pi]^d + i \Phi (x,y)} e^{\frac{i}{\ep}(y-x) \tilde{\xi}} q( x, y, \tilde{\xi} - i \Phi (x,y); \ep) 
\, d\tilde{\xi} \nonumber\\
&= (2\pi)^{-d} \int_{[-\pi, \pi]^d} e^{\frac{i}{\ep}(y-x) \tilde{\xi}} q( x, y, \tilde{\xi} - i \Phi (x,y); \ep) \, 
d\tilde{\xi} 
\label{3prop3app}
\end{align}
The right hand side of \eqref{3prop3app} is the integral kernel of $\widetilde{\Op}^\T_\ep \bigl(\hat{q}_\psi\bigr)$ for
$\hat{q}_\psi$ given by \eqref{00prop3app}. 
Since all derivatives of $\Phi$ are bounded by assumption, if follows that
$\hat{q}_\psi\in S^k_\delta \bigl( 1\bigr)\bigl(\R^{2d}\times \T^d\bigr)$. The statement on the analyticity of $\hat{q}_\psi$ with respect to $\xi$ and
on the existence of an asymptotic expansion follow at once from equality \eqref{00prop3app}.

Concerning the statement on the $t$-quantization we use Proposition \ref{prop2app}, showing that there is a unique
symbol $\hat{q}_{t,\psi} \in S^k_\delta \bigl( 1\bigr)\bigl(\R^{d}\times \T^d\bigr)$ such that 
$\widetilde{\Op}_\ep^\T (\hat{q}_\psi) = \Op_{\ep, t}^\T (\hat{q}_{t,\phi})$. Moreover, by \eqref{2prop2app},  we have in leading order, i.e. 
modulo $S_\delta^{k+1-2\delta}(1)$,
\begin{equation}\label{4prop3app}
\hat{q}_{t,\psi} (x, \xi; \ep) = \hat{q}_\psi(x, x, \xi; \ep) = q (x, x, \xi - i \Phi(x, x); \ep) = 
q (x, x, \xi - i \nabla \psi(x); \ep)\; .
\end{equation} 
and $q_{\psi,t}$ has an asymptotic expansion with the stated properties.

\end{proof}

\begin{rem}\label{remprop3app}
Let $p\in S^k_\delta \bigl( 1\bigr)\bigl(\R^{d}\times \T^d\bigr)$ and $s, t\in [0,1]$. 
Then it follows at once from Remark \ref{remprop1app} that
$e^{\psi/\ep} \widetilde{\Op}^\T_{\ep, t}(p) e^{-\psi / \ep}$ is the $s$-quantization of a symbol 
$p_{\psi, s}\in  S^k_\delta \bigl( 1\bigr)\bigl(\R^{d}\times \T^d\bigr)$ satisfying
\[ p_{\psi, s}(x, \xi; \ep) = p(x, \xi - i\nabla \psi (x); \ep) \mod S^{k+1}_\delta(1) \; . \]
\end{rem}

\section{Former results}\label{app2}

In the more general setting, that there might be more than two Dirichlet operators with spectrum inside of the
spectral interval $I_\ep$, let
\begin{eqnarray}\label{specHepusw}
\spec (H_\ep) \cap I_\ep = \{ \lambda_1,\ldots , \lambda_N\} \,
,&\quad&
u_1,\ldots ,u_N\in \ell^2\left(\disk\right)\\
\F := \Span \{u_1,\ldots u_N\} \nonumber\\
\spec \left(H_\ep^{M_j}\right) \cap I_\ep = \{ \mu_{j,1},\ldots,
\mu_{j,n_j}\} \, ,
&\quad& v_{j,1},\ldots,v_{j,n_j}\in \ell^2_{M_{j,\ep}},\, j\in {\mathcal C} \nonumber\\
\E_j := \Span \{ v_{j,1},\ldots, v_{j,n_j} \} \, , &\quad & \E :=
\bigoplus \E_j \nonumber
\end{eqnarray}
denote the eigenvalues of $H_\ep$ and of the Dirichlet operators
$H_\ep^{M_j}$ defined in \eqref{HepD} inside the spectral interval  $I_\ep$ and the corresponding real orthonormal systems
of eigenfunctions (these exist because all operators commute with complex conjugation). We write 
\begin{equation}\label{valpha}
v_\alpha\quad\text{with}\quad \alpha =(\alpha_1, \alpha_2)\in \mathcal{J}:=\{(j,k)\,|\,j\in\mathcal{C},\, 1\leq k \leq n_j\}
\quad \text{and}\quad j(\alpha):= \alpha_1\, .
\end{equation}
We remark that the number of eigenvalues $N, n_j\, ,\, j\in\mathcal{C}$ with respect to $I_\ep$ as defined in 
\eqref{specHepusw} may depend on $\ep$.

For a fixed spectral interval $I_\ep$, it is shown in \cite{kleinro4} that the distance
$\vec{\dist}(\E, \F):= \| \Pi_\E - \Pi_\F \Pi_\E\|$ is
exponentially small and determined by $S_0$, the Finsler distance between the two nearest neighboring wells.

The following theorem, proven in \cite{kleinro4}, gives the representation of $H_\ep$ restricted to an
eigenspace with respect to the basis of Dirichlet eigenfunctions.

\begin{theo}\label{ealphafalpha}
In the setting of Hypotheses \ref{hyp1}, \ref{hypIMj} and \eqref{specHepusw}, \eqref{valpha}, set
$\mathcal{G}_v:=\left(\skpd{v_\alpha}{v_\beta}\right)_{\alpha,\beta\in\mathcal{J}}$, the Gram-matrix, and
$\vec{e}:=\vec{v}\mathcal{G}_v^{-\frac{1}{2}}$, the orthonormalization of
$\vec{v}:=(v_{1,1}.\ldots, v_{m,n_m})$. Let $\Pi_\F$ be the orthogonal projection onto $\F$ and set $f_\alpha=\Pi_{\F} e_\alpha$. For
$\mathcal{G}_f=\left(\skpd{f_\alpha}{f_\beta}\right)$, we
choose
$\vec{g}:=\vec{f}\mathcal{G}_f^{-\frac{1}{2}}$ as orthonormal basis of $\F$.

Then there exists $\ep_0>0$ such that for all $\sigma < S$ and $\ep\in (0,\ep_0]$ the following holds.
\ben
\item The matrix of $H_\ep|_{\F}$ with respect to $\vec{g}$
is given by
\[  \diag \bigl(\mu_{1,1}, \ldots, \mu_{m,n_m}\bigl) +
\left(\tilde{w}_{\alpha,\beta}\right)_{\alpha,\beta\in\mathcal{J}} +
O\left(e^{-\frac{2\sigma}{\ep}}\right) \]
where
\[ \tilde{w}_{\alpha,\beta} = \frac{1}{2}(w_{\alpha\beta}+ w_{\beta\alpha}) =
O\left(\ep^{-N}e^{-\frac{d(x_{j(\alpha)},x_{j(\beta)})}{\ep}}\right)  \]
with
\begin{equation}\label{interact}
w_{\alpha,\beta} = \skpd{v_\alpha}{(\id -\id_{M_{j(\beta)}})T_\ep v_\beta} =
\sum_{\natop{x\in\disk}{x\notin M_{j(\beta)}}}
\sum_{\gamma\in\disk}
a_\gamma(x; \ep)v_\beta(x+\gamma) v_\alpha (x)
\end{equation}
and
$\tilde{w}_{\alpha,\beta} = 0$ for $j(\alpha) = j(\beta)$. The remainder $O\bigl(e^{-\frac{2\sigma}{\ep}}\bigr)$ is estimated 
with respect to the operator norm.
\item
There exists a bijection 
\[ b:\spec (H_\ep|_{\F})\ra \spec
\bigl((\mu_\alpha \delta_{\alpha\beta} +
\tilde{w}_{\alpha\beta})_{\alpha,\beta\in\mathcal{J}}\bigr)\quad\text{such that}\quad |b(\lambda) - \lambda|
= \expord{-2\sigma}\]
where the eigenvalues are counted with multiplicity.
\een
\end{theo}

\end{appendix}

\end{document}